\documentclass[11pt]{article} 

\pdfoutput=1

\usepackage{mystyle}
\usepackage{url}
\usepackage{LatexDefinitions}

\newcommand{\mnorm}{\mc{N}}

\newcommand{\trucos}{{\ol{\cos}}}
\newcommand{\dens}[2]{\textstyle{\frac{#1}{#2}}}
\numberwithin{equation}{section}

\begin{document}
\title{Unbalanced Optimal Transport:\\ Dynamic and Kantorovich Formulations}

\date{}
\author{
\begin{tabular}{c}
	L\'ena\"ic Chizat \qquad Gabriel Peyr\'e\\
	Bernhard Schmitzer  \qquad Fran\c{c}ois-Xavier Vialard\\[2mm]
	Ceremade, Universit\'e Paris-Dauphine \\
	\texttt{\small\{chizat,peyre,schmitzer,vialard\}@ceremade.dauphine.fr}
\end{tabular} }


\maketitle



\begin{abstract}
This article presents a new class of distances between arbitrary nonnegative Radon measures inspired by optimal transport. These distances are defined by two equivalent alternative formulations: (i) a \emph{dynamic} formulation defining the distance as a geodesic distance over the space of measures (ii) a static ``Kantorovich'' formulation where the distance is the minimum of an optimization problem over pairs of couplings describing the transfer (transport, creation and destruction) of mass between two measures. Both formulations are convex optimization problems, and the ability to switch from one to the other depending on the targeted application is a crucial property of our models. 
Of particular interest is the Wasserstein-Fisher-Rao metric recently introduced independently by~\cite{ChizatOTFR2015,new2015kondratyev}. Defined initially through a dynamic formulation, it belongs to this class of metrics and hence automatically benefits from a static Kantorovich formulation. 
%
\end{abstract}


\section{Introduction}

Optimal transport is an optimization problem which gives rise to a popular class of metrics between probability distributions. We refer to the monograph of Villani~\cite{cedric2003topics} for a detailed overview of optimal transport. 
A major constraint of the resulting transportation metrics is that they are restricted to measures of equal total mass (e.g.\ probability distributions). In many applications, there is however a need to compare unnormalized measures, which corresponds to so-called \emph{unbalanced} optimal transport problems, following the terminology introduced in~\cite{benamou2003numerical}. Applications of these unbalanced metrics range from image classification~\cite{rubner1997earth,pele2008linear} to the processing of neuronal activation maps~\cite{GramfortMICCAI}.
This class of problems requires to precisely quantify the amount of transportation, creation and destruction of mass needed to compare arbitrary positive measures. While several proposals to achieve this goal have been made in the literature (see below for more details), to the best of our knowledge, there lacks a coherent framework that enables to deal with generic measures while preserving both the dynamic and the static perspectives of optimal transport. It is precisely the goal of the present paper to describe such a framework and to explore its main properties. 

\subsection{Previous work}
In the last few years, there has been an increasing interest in extending optimal transport to the unbalanced setting of measures having non-equal masses. 
\paragraph{Dynamic formulations of unbalanced optimal transport. }

Several models based on the fluid dynamic formulation introduced in~\cite{benamou2000computational} have been proposed recently~\cite{OTmaasrumpf,lombardi2013eulerian,piccoli2014generalized,piccoli2013properties}. In these works, a source term is introduced in the continuity equation. They differ in the way this source is penalized or chosen.
We refer to~\cite{ChizatOTFR2015} for a detailed overview of these models.  

\paragraph{Static formulations of unbalanced optimal transport. }
Purely static formulations of unbalanced transport are however a longstanding problem. 
A simple way to address this issue is given in the early work of Kantorovich and Rubinstein~\cite{kantorovich1958space}. The corresponding ``Kantorovich norms'' were later extended to separable metric spaces by~\cite{hanin1999extension}. These norms handle mass variations by allowing to drop some mass from each location with a fixed transportation cost. 
The computation of these norms can in fact be re-casted as an ordinary optimal transport between normalized measures by adding a point ``at infinity'' where mass can be sent to, as explained by~\cite{guittet2002extended}. This reformulation is used in~\cite{GramfortMICCAI} for applications in neuroimaging. 
A related approach is the so-called optimal partial transport. It was initially proposed in the computer vision literature to perform image retrieval \cite{rubner1997earth,pele2008linear}, while its mathematical properties are analyzed in detail by \cite{caffarelli2010free,figalli2010optimal}. 
As noted in~\cite{ChizatOTFR2015} and recalled in Section \ref{sec:ExamplesPartial}, optimal partial transport is tightly linked to the generalized transport proposed in~\cite{piccoli2014generalized,piccoli2013properties} which allows a dynamic formulation of the optimal partial transport problem.
The contributions in \cite{piccoli2014generalized} were inspired by \cite{benamou2003numerical} where it is proposed to relax the marginal constraints and to add an $L^2$ penalization term instead.
%
\paragraph{Wasserstein-Fisher-Rao metric and relation with recent work.}

A new metric between measures of non equal masses has recently and independently been proposed by~\cite{ChizatOTFR2015,new2015kondratyev}. This new metric interpolates between the Wasserstein $W_2$ and the Fisher-Rao (also known as Hellinger) metrics. It is defined through a dynamic formulation, corresponding formally to a Riemannian metric on the space of measures, which generalizes the formulation of optimal transport due to Benamou and Brenier~\cite{benamou2000computational}. 

In~\cite{ChizatOTFR2015}, we proved existence of geodesics, presented the limit models (for extreme values of mass creation/destruction cost) and proposed a numerical scheme based on first order proximal splitting methods. We also thoroughly treated the case of two Dirac masses which was the first step towards a Lagrangian description of the model. This metric is the prototypical example for the general framework developed in this article. It thus enjoys both a dynamic formulation and a static one (Sect.~\ref{sec:equivalence WF}). 

\subsection{Contribution}
\label{sec:IntroContribution}
The initial motivation of this article is studying of the  Wasserstein-Fisher-Rao ($\WF$) metric. For two non-negative densities $\rho_0$, $\rho_1$ on a domain $\Omega \subset \R^d$ it is informally obtained as
\begin{align}
	\label{eq:WFInformal}
	\WF^2(\rho_0,\rho_1) \eqdef \inf_{(\rho,v,\alpha)} \int_0^1 \int_{\Omega} \left( \frac{1}{2} |v_t(x)|^2 + \frac{1}{2} g_t(x)^2 \right)\,\d \rho_t(x)\,\d t
\end{align}
where $(\rho_t)_{t\in [0,1]}$ is a time-dependent density that interpolates between $\rho_0$ and $\rho_1$, $(v_t)_{t\in [0,1]}$ is a velocity field that describes the movement of mass and $(g_t)_{t\in [0,1]}$ a scalar field that models local growth and destruction of mass. They must together satisfy (distributionally) the continuity equation with source
\begin{align}
	\label{eq:ContinuityEqInformal}
	\partial_t \rho_t + \nabla \cdot (\rho_t\,v_t) = \rho_t\,g_t.
\end{align}

%
%

In Section \ref{sec:dynamic}, we generalize this model and look at a wider family of dynamic problems given by
\begin{align}
	\label{eq:DynamicInformal}
	C_D(\rho_0,\rho_1) = \inf_{(\rho,v,\alpha)} \int_0^1 \int_{\Omega} f\big(x,\rho_t(x),v_t(x) \cdot \rho_t(x), g_t(x) \cdot \rho_t(x) \big)\,\d x\,\d t
\end{align}
where the infimum is again taken over solutions of \eqref{eq:ContinuityEqInformal}. Here, for $(\rho,v,g)\in \R_+\times \R^d\times \R$, $f(x,\rho,v \cdot \rho, g \cdot \rho)$ gives the \emph{infinitesimal cost} of moving  a particle of mass $\rho$ at $x$ in direction $v$ while undergoing an infinitesimal scaling by the (signed) rate of growth $g$.
Note that in the two last arguments of $f$ we multiply $v$ and $\alpha$ by $\rho$. This corresponds to the velocity to momentum change of variables proposed in \cite{benamou2000computational} to obtain a convex problem. We show that this problem admits a dual formulation (Proposition~\ref{prop: dynamic dual}) and allows to define geodesic metrics on the space of nonnegative measures (Proposition~\ref{prop:metricproperty}).

Our main goal in this article is to propose a Kantorovich-like formulation for this family of problems, where the time variable is removed. This requires to adapt the definitions as done in Section~\ref{sec : general static problem}.  We introduce a new Kantorovich-like class of \emph{static} problems of the form
\begin{align}
	\label{eq:StaticInformal}
	C_K(\rho_0,\rho_1) = \inf_{(\gamma_0,\gamma_1)} \int_{\Omega \times \Omega} c(x,\gamma_0(x,y),y,\gamma_1(x,y))\,\d x\,\d y
\end{align}
where $(\gamma_0,\gamma_1)$ are two \emph{semi-couplings} between $\rho_0$ and $\rho_1$, describing analogously to standard optimal transport, how much mass is transported between any pair $x$, $y \in \Omega$. Two semi-couplings are required, to be able to describe changes of mass during transport. The function $c(x_0,m_0,x_1,m_1)$ determines the cost of transporting a quantity of mass $m_0$ from $x_0$ to a (possibly different) quantity $m_1$ at $x_1$.
It is a crucial assumption of our approach that $c(x,\cdot,y,\cdot)$ is jointly positively 1-homogeneous and convex in the two mass arguments. This ensures that \eqref{eq:StaticInformal} can be rigorously defined as an optimization problem over measures and that the resulting problem is convex.
A dual problem is established (Proposition \ref{prop: duality}).
Analogous to standard optimal transport, when $c$ induces a metric over pairs of location and mass, then \eqref{eq:StaticInformal} defines a metric over the space of nonnegative measures (Theorem \ref{th:metric}).

In Section \ref{sec:DynamicToStatic}, under suitable assumptions on $f$, we establish equivalence between \eqref{eq:StaticInformal} and \eqref{eq:DynamicInformal} when $c$ is chosen to be the \emph{minimal path} cost induced by $f$ (Theorem \ref{th: continuous static general} and Proposition \ref{prop : alternative dynamic/static}). This is our main result, which is analogous to the Benamou-Brenier formula for classical optimal transport.

Finally, we apply those results to two unbalanced optimal transport models. 
Section \ref{sec:ExamplesPartial} introduces a dynamic formulation and gives duality results for a family of metrics obtained from the optimal partial transport problem. This is reminiscent of --- and generalizes --- the results in \cite{piccoli2013properties,piccoli2014generalized}.
The case of the $\WF$ metric is treated in Section \ref{sec:equivalence WF}. In Section \ref{sec:StaticGamma}, it is shown how standard static optimal transport is obtained as a limit (in the sense of $\Gamma$-convergence) of the $\WF$ metric, thus complementing a previous result of~\cite{ChizatOTFR2015} obtained for dynamic formulations. 

\subsection[Relation with Liero,Mielke,Savare 2015]{Relation with \cite{LieroMielkeSavareLong,LieroMielkeSavareShort}}

After completing this paper, we became aware of the independent work of \cite{LieroMielkeSavareLong,LieroMielkeSavareShort}. In these two papers the authors develop and study a similar class of ``static'' transportation-like problems as here. This huge body of work contains many theoretical aspects that we do not cover. For instance measures defined over more general metric spaces are considered, while we work over $\R^d$.
The construction of~\cite{LieroMielkeSavareLong} defines three equivalent static formulations. Their third ``homogeneous'' formulation is closely related (by a change of variables) to our ``semi-couplings'' formulation. Their first formulation gives an intuitive and nice interpretation of this class of convex programs as a modification of the original optimal transportation problem where one replaces the hard marginal constraints by soft penalization using  Csisz\'ar $f$-divergences. The dual of this first formulation is related to the dual of our formulation by a logarithmic change of variables, see Corollary \ref{cor : static log-entropic}. Quite interestingly, the same idea is used in an informal and heuristic way by~\cite{FrognerNIPS} for applications in machine learning, where soft marginal constraints is the key to stabilize numerical results. 
The authors of~\cite{LieroMielkeSavareLong,LieroMielkeSavareShort} study dynamical formulations in the Wasserstein-Fisher-Rao setting (that they call the ``Hellinger-Kantorovich'' problem). This allows them to make a detailed analysis of the geodesic structure of this space. In contrast, we study a more general class of dynamical problems, but restrict our attention to the equivalence with the static problem. Another original contribution of our work is the proof of the metric structure (in particular the triangular inequality) for static and dynamic formulations when the underlying cost over the cone manifold $\Omega \times \R^+$ is related to a distance. 
%

\subsection{Preliminaries and notation}
We denote by $C(X)$ the Banach space of real valued continuous functions on a compact set $X\subset \R^d$ endowed with the sup norm topology. Its topological dual is identified with the set of Radon measures, denoted by $\mathcal{M}(X)$ and the dual norm on $\mathcal{M}(X)$ is the total variation, denoted by $|\cdot|_{TV}$. Another useful topology on $\mathcal{M}(X)$ is the weak* topology arising from this duality: a sequence of measures $(\mu_n)_{n\in \N}$ weak* converges towards $\mu \in \mathcal{M}(X)$ if and only if for all $u \in C(X)$, $\lim_{n\rightarrow + \infty}\int_{X} u \, \d\mu_n = \int_{X} u \,\d\mu$. According to that topology, $C(X)$ and $\mathcal{M}(X)$ are topologically paired spaces (the elements of each space can be identified with the continuous linear forms on the other), this is a standard setting in convex analysis.
We also use the following notations:
\begin{itemize}
	\item $\mathcal{M}(X)$ is the vector space of Radon measures and $\mathcal{M}_+(X)$ the cone of \emph{nonnegative} Radon measures;

	\item $\mu \ll \nu$ means that the $\R^m$-valued measure $\mu$ is absolutely continuous w.r.t.\ the positive measure $\nu$. We denote by $\dens{\mu}{\nu} \in (L^1(X,\nu))^m$ the density of $\mu$ with respect to $\nu$;

	\item Ffr a (possibly vector) measure $\mu$, $|\mu| \in \mathcal{M}_+(X)$ is its variation;

	\item $\pfwd{T} \mu$ is the image measure of $\mu$ through the measurable map $T: X_1 \to X_2$, also called the pushforward measure. It is given by $\pfwd{T} \mu (A_2) \eqdef \mu(T^{-1}(A_2))$;

	\item $\delta_x$ is a Dirac measure of mass $1$ located at the point $x$;

	\item $\iota_{\mathcal{C}}$ is the (convex) indicator function of a convex set $\mathcal{C}$ which takes the value $0$ on $\mathcal{C}$ and $+\infty$ everywhere else;

	\item if $(E,E')$ are topologically paired spaces and $f : E \to \R \cup \{+ \infty\}$ is a convex function, $f^*$ is its Legendre transform i.e.\ for $y \in E'$, $f^*(y) \eqdef \sup_{x\in E} \langle x, y \rangle - f(x) $; 

	\item for $n \in \N$ and a tuple of distinct indices $(i_1,\ldots,i_k)$, $i_l \in \{0,\ldots,n-1\}$ the map
\begin{equation}
	\Proj_{i_1,\ldots,i_k} : \Omega^n \rightarrow \Omega^k
\end{equation}
denotes the canonical projection from $\Omega^n$ onto the factors given by the tuple $(i_1,\ldots,i_k)$;

%

	\item the truncated cosine is defined by $\trucos : z \mapsto \cos(|z|\wedge \frac{\pi}{2})$;
	\item we denote by $f\wedge g$ the biggest function (or measure) that is smaller than both $f$ and $g$. This is to be contrasted with the notation $\min S$ which for a totally ordered set $S$ denotes its smallest element.
\end{itemize}



\section{A Family of Dynamic Problems}\label{sec:dynamic}
In this section, we describe a first approach to unbalanced optimal transport, which generalizes~\eqref{eq:WFInformal} and is inspired by the \emph{dynamic} formulation of classical optimal transport. 
%
Dynamic formulations of unbalanced transport models correspond intuitively to the computation of geodesic distances according to a function 
measuring the infinitesimal effort needed for ``acting'' on a mass $\rho$ at position $x$ according to the speed $v$ and rate of growth $g$ (cf.~\eqref{eq:DynamicInformal}). In this section, $\Omega$ is the closure of an open, connected, bounded subset of $\R^d$ with Lipschitz boundary.

\subsection{Continuity equation}
The continuity equation, informally introduced in \eqref{eq:ContinuityEqInformal}, is a key concept for the dynamic formulations of this article. It enforces a local mass preservation constraint for a density $\rho$, a flow field $v$ and a growth rate field $g$.
We now give a rigorous definition in terms of measures $(\rho,\omega,\zeta)$ where $\omega$ can informally be interpreted as the momentum $\rho \cdot v$ of the flow field and $\zeta$ corresponds to $\rho \cdot g$. As opposed to what is standard in the literature, we do not require a priori $\omega$ and $\zeta$ to have a density with respect to $\rho$ : this allows to deal with a wider class of action functionals.
\begin{definition}[Continuity equation with source]
\label{def:continuity equation}
For $(a,b)\in \R^2$ and a compact domain $\Omega \subset \R^d$, denote by $\mathcal{CE}_a^b(\rho_0,\rho_1)$ the affine subset of $\mathcal{M}([a,b]\times \Omega)\times \mathcal{M}([a,b]\times \Omega)^d \times \mathcal{M}([a,b]\times \Omega)$ of triplets of measures $\mu=(\rho,\omega,\zeta)$ satisfying the continuity equation 
\[
\partial_t \rho + \nabla \cdot \omega = \zeta
\]
 in the distributional sense, interpolating between $\rho_0$ and $\rho_1$ and satisfying homogeneous Neumann boundary conditions. More precisely we require
\begin{equation}
\label{eq:continuity weak}
\int_a^b \int_{\Omega} \partial_t \varphi \ \d\rho + \int_a^b \int_{\Omega} \nabla \varphi \cdot \d\omega + \int_a^b \int_{\Omega} \varphi \ \d\zeta = \int_{\Omega} \varphi(b,\cdot)\d\rho_1 - \int_{\Omega} \varphi(a,\cdot)\d\rho_0
\end{equation}
for all $\varphi \in C^1([a,b]\times \Omega)$.
\end{definition}
Below we collect two simple facts on this equation which will turn out useful. Their proof only involves elementary operations that we do not reproduce here for conciseness. The notation $B^d(0,r)$ denotes the open ball of radius $r$ in $\R^d$ centered at the origin.
\begin{proposition}
\label{prop:results on continuity equation}
\hfill
\begin{description}
%
\item[Smoothing.] Let $\veps>0$, let $r^x_{\veps}$, $r^t_\veps$ mollifiers supported on the open balls $B^d(0,\frac{\veps}{2})$ and $B^1(0,\frac{\veps}{2})$ respectively and $r_{\veps} : (t,x) \mapsto r_{\veps}^t(t)r_{\veps}^x(x)$. Let $\mu = (\rho,\omega,\zeta)$ be a triplet of measures supported on $\R\times \Omega$ such that $\mu \in \mathcal{CE}_0^1(\rho_0,\rho_1)$, $\mu = (\rho_0,0,0)\otimes \d t$ for $t<0$, and $\mu = (\rho_1,0,0)\otimes \d t$ for $t>1$. Then for all $a\leq-\veps/2$, and $b\geq1+\veps/2$, $\mu \ast r_\veps \in \mathcal{CE}_a^b(\rho_0 \ast r_\veps^x,\rho_1 \ast r_\veps^x)$ on $\Omega+ \bar{B}^d(0,\veps/2)$.
\item[Scaling.] Let $\mu = (\rho,\omega,\zeta) \in \mathcal{CE}_a^b(\rho_a,\rho_b)$ with $a<b$ and $T : (t,x) \mapsto (T_t(t), T_x(x))$ be an affine scaling with multiplication factor $\alpha$ and $\beta$, respectively. Then $(\alpha \cdot T_\#\rho, \beta \cdot T_\# \omega,  T_\#\zeta) \in \mathcal{CE}_{T_t(a)}^{T_t(b)}((T_x)_\#(\rho_a),(T_x)_\#(\rho_b))$ on the domain $T_x(\Omega)$.
\end{description}
\end{proposition}


%
\subsection{Action minimizing problems}
In order to select an interpolation among all the solutions to the continuity equation, we choose that which minimizes an \emph{action} functional. A crucial feature of this action which allows to adapt results from classical optimal transport theory is $1$-homogeneity with respect to mass. It is necessary for the model to behave similarly for diffuse and discrete measures. 

Now, two choices of variables are possible. A first approach is to choose a Lagrangian function $L: \Omega \times \R^d \times \R\to \R$ which is a function of the position, the velocity and the rate of growth and to integrate it in space and time to define the action
\begin{equation}\label{eq:actionlagrangian}
\int_0^1 \int_\Omega L(x,v_t(x),g_t(x))\d \rho_t(x),
\end{equation}
which directly generalizes~\eqref{eq:WFInformal}. However, functions built this way are generally not convex even if $L$ is, nor is the continuity equation an affine constraint. This issues are solved if one consider the other set of variables $(\rho,\omega,\zeta)$ because (i) the continuity constraint (Definition~\ref{def:continuity equation}) is affine and (ii) the function $L(x,\omega/\rho,\zeta/\rho)\rho$ is convex and positively $1$-homogeneous if $L$ is convex (this transformation corresponds to taking the \emph{perspective function} of $L$).
In the following, we thus directly consider such convex functions of $(\rho,\omega,\zeta)$ that we refer to as \emph{infinitesimal costs}.



\begin{definition}[Infinitesimal cost]
\label{def: infinitesimal cost}
In this paper, an infinitesimal cost is a lower semicontinuous function $f: \Omega \times \R \times \R^d \times \R \to [0,+\infty]$ such that for all $x\in \Omega$, $f(x,\cdot,\cdot,\cdot)$ is convex, positively 1-homogeneous and satisfies
\[
f(x,\rho,\omega,\zeta) 
\begin{cases}
= 0 & \tn{if } (\omega,\zeta) = (0,0) \tn{ and } \rho \geq 0\\
\in ]0,+\infty[ & \tn{if } (\omega,\zeta) \neq (0,0) \text{ and } \rho>0 \\
= + \infty & \tn{if } \rho<0 \, .
\end{cases}
\]
\end{definition}

It is clear that any continuous, convex Lagrangian defines an infinitesimal cost and that reciprocally an infinitesimal cost $f$ defines a continuous, convex Lagrangian $L(x,v,g)= f(x,1,v,g)$ so the two approaches (starting from a Lagrangian or an infinitesimal cost) are in fact equivalent. 

\begin{example}
With the choice of Lagrangian $L(v,g)=\frac12 (|v|^2+g^2)$, the associated infinitesimal cost $f$ is the lower-semicontinuous relaxation of $\frac12 (|\omega|^2/\rho + \zeta^2/\rho)$. This case correspond to the metric $\WF$. We also study in Section~\ref{sec:Examples} Lagrangians of the form $L(v,g)=\frac1p|v|^p+|g|$ with $p>1$, which correspond to the infinitesimal cost $\frac1p |\omega|^p/\rho^{p-1}+|\zeta|$ and show the connection with optimal partial transport.
\end{example}

The dynamic formulation is defined as the minimization of an action defined from the infinitesimal cost $f$.

\begin{definition}[Dynamic problem]
\label{def:DynamicProblem}
For $(\rho,\omega,\zeta) \in \mathcal{M}([0,1] \times \Omega)^{1+d+1}$, let
\begin{align}\label{eq:Jdynamic}
	J_D(\rho,\omega,\zeta) & \eqdef \int_0^1 \int_{\Omega} f(x,\dens{\rho}{\lambda},\dens{\omega}{\lambda},\dens{\zeta}{\lambda})\, \d\lambda(t,x)
	\intertext{where $\lambda \in \mc{M}_+([0,1] \times \Omega)$ is such that $(\rho,\omega,\zeta) \ll \lambda$. Due to 1-homogeneity of $f$, this definition does not depend on the choice of $\lambda$. The dynamic problem is, for $\rho_0,\rho_1 \in \mathcal{M}_+(\Omega)$,}
	\label{eq:dynamic problem}
	C_D(\rho_0,\rho_1) & \eqdef \inf_{(\rho,\omega,\zeta) \in \mathcal{CE}_0^1(\rho_0,\rho_1)} J_D(\rho,\omega,\zeta)\,.
\end{align}
\end{definition}

The expression in~\eqref{eq:Jdynamic} is quite abstract due to the dummy reference measure, which allows to deal correctly with the singular parts of the measures, but it is mainly a rewritting of
~\eqref{eq:actionlagrangian} in terms of the variables $(\rho,\omega,\zeta)$. We will sometimes make the following restrictive assumptions on the infinitesimal cost $f$.

\begin{enumerate}[{(C}1{)}]
\item \emph{multiplicative dependency on $x$}: there exist continuous functions $\lambda_i : \Om \to ]0,+\infty[, \, i\in \{1,\dots , N\}$ such that 
\begin{equation}\label{continuity assumption}
f(x,\rho,\omega,\zeta) = \sum_{i=1}^N \lambda_i(x) \tilde{f}_i(\rho,\omega,\zeta)\, .
\end{equation}
\label{ass: multiplicative}
%
\item \emph{doubling condition}: there exists $C>0$ such that $f(x,\rho,\omega,2\zeta)\leq C \cdot f(x,\rho,\omega,\zeta)$, for all $(x,\rho,\omega,\zeta) \in \Omega \times \R \times \R^d \times \R$.
\label{ass: doubling}
 \end{enumerate}
 %
As is shown in the next result, assumption (C\ref{ass: doubling}) is enough for the dynamic cost $C_D$ to be finite between any pair of nonnegative measures.
\begin{proposition}[Finite cost]
\label{prop: finite cost dynamic}
Let $f$ be an infinitesimal cost satisfying assumption (C\ref{ass: doubling}) and such that  $f(x,1,0,1)$ is bounded on $\Omega$. Then for all $(\rho_0,\rho_1)\in \mathcal{M}_+(\Omega)^2$, the dynamic cost \eqref{eq:dynamic problem} is finite, i.e.\ $C_D(\rho_0,\rho_1)< +\infty$.
\end{proposition}
\begin{proof}
Let $C$ be the constant of assumption (C\ref{ass: doubling}) and first assume that $\rho_0=0$. Define, for some $\alpha>0$, the measures $\rho = \rho_1 \otimes (t^\alpha \d t)$, $\omega = 0$ and $\zeta = \rho_1 \otimes (\alpha t^{\alpha-1} \d t)$. By construction $(\rho,\omega,\zeta)\in \mathcal{CE}_0^1(0,\rho_1)$, it only remains to show that for $\alpha$ big enough this triplet has finite cost. By homogeneity of $f$:
\begin{align*}
J_D(\rho,\omega,\zeta) &=
\int_0^1 t^{\alpha} \d t \int_\Omega \d\rho_1 f(x,1,0,\alpha /t)\, .
\end{align*}
By convexity, $f(x,1,0,\cdot)$ increases with the module of the last argument, so with $k(t) := \lceil \log_2(\alpha/t) \rceil$ it holds 
\[
0 \leq f(x,1,0,\alpha/t) \leq f(x,1,0,2^{k(t)}) \leq C^{k(t)} f(x,1,0,1) \leq C' t^{-\log_2 C}
\]
for some $C'>0$ which does not depend on $t$. Consequently, for $\alpha >\log_2 C -1$,
\begin{align*}
C_D(0,\rho_1) \leq J_D(\rho,\omega,\zeta) \leq
C' \int_0^1 t^{\alpha-\log_2 C} \d t \int_\Omega \d\rho_1  < +\infty\, .
\end{align*}
Remark that the preceding argument works as well when reverting or shortening the time interval. Thus, if $\rho_0$ is not the null measure, one builds a interpolation by first decreasing the mass from $\rho_0$ to $0$ for $t\in [0,1/2]$ and then increasing from $0$ to $\rho_1$ for $t\in [1/2,1]$, and the cost of this interpolation is again finite.
\end{proof}

\subsection{Properties of dynamic formulations}
The dynamic problem enjoys a dual formulation,  which is the maximization of a linear objective over the set of sub-solutions to a Hamilton-Jacobi equation. Unlike the dual formulation of classical optimal transport, this Hamilton-Jacobi equation involves a term of zeroth order. We state here the proposition in a rather abstract form and explicit examples are given in Section~\ref{sec:Examples}.

\begin{proposition}[Duality]
\label{prop: dynamic dual}
If $f$ is an infinitesimal cost satisfying assumption (C\ref{ass: multiplicative})  then the minimum of \eqref{eq:dynamic problem} is attained. Moreover, denoting $B(x)$ the polar set of $f(x,\cdot,\cdot,\cdot)$ for all $x\in \Omega$, it holds
\begin{equation}
\label{eq: dynamic primal problem}
C_D (\rho_0,\rho_1) = \sup_{\varphi \in K} \int_\Omega \varphi(1,\cdot) \d \rho_1 - \int_\Omega \varphi(0,\cdot) \d \rho_0
\end{equation}
with 
$
K \eqdef \left\{ \varphi \in  C^1([0,1]\times \Omega) : (\partial_t \varphi, \nabla \varphi, \varphi)\in B(x), \, \forall (t,x)\in[0,1]\times \Omega \right\} \, .
$
\end{proposition}

\begin{corollary}[Sublinearity]\label{cor:sublinearityCD}
As the supremum of continuous linear functionals, $C_D$ is a sublinear (i.e.\ convex, positively $1$-homogeneous) and weakly* lower semicontinuous functional.
\end{corollary}

\begin{proof}
Remark, that \eqref{eq: dynamic primal problem} can be written as 
\[
- \inf_{\varphi \in C^1([0,1] \times \Omega)} F(A\varphi) + G(\varphi)
\]
where
$A : \varphi \mapsto (\partial_t \varphi, \nabla \varphi, \varphi)$, is a bounded linear operator from $C^1([0,1]\times \Omega)$ to $C([0,1]\times \Omega)^{d+2}$, and
$F :(\alpha,\beta,\gamma) \mapsto \int_0^1 \int_{\Omega} \iota_{B(x)}(\alpha(t,x),\beta(t,x),\gamma(t,x)) \d x \d t $,
$ G : \varphi \mapsto \int_{\Omega} \varphi(0,\cdot)\d\rho_0 - \int_{\Omega} \varphi(1,\cdot) \d\rho_1$ are convex, proper and lower-semi\-continuous functionals, in particular because for all $x\in \Omega$, the set $B(x)$ is convex, closed and contains $0_{\R^{d+2}}$.
Since we assumed that $f(x,\rho,\omega,\zeta)>0$ if $|\omega|>0$ or $\zeta>0$ and $f$ is continuous as a function of $x$ on the compact $\Omega$, one can check that there exists $\veps>0$ such that $(-\veps,0,\theta \epsilon/2) \in \left( \tn{ int } \cap_{x\in \Omega} B(x)\right)$ for $\theta\in [-1,1]$ and thus the function $\varphi : t \mapsto -\veps t +\veps/2$ is such that $F(A\varphi) + G(\varphi)<+\infty$ and $F$ is continuous at $A\varphi$. Then, by Fenchel-Rockafellar duality, \eqref{eq: dynamic primal problem} is equal to
\[
\min_{\mu \in \mathcal{M}([0,1]\times \Omega)^{d+2}}  G^*(-A^*\mu) + F^*(\mu) \, .
\]
By Lemma~\ref{lem:dualitymeasures}, we have $F^*=J_D$, and by direct computations, $G^*\circ (-A^*)$ is the convex indicator of $\mathcal{CE}_0^1(\rho_0,\rho_1)$. 
\end{proof}

The lower-semicontinuity and duality results in this article rely on the following duality property of integral functionals of measures. It is a rephrasing of \cite[Theorem 6]{rockafellar1971integrals} with simplified assumptions thanks to \cite[Lem. A.2]{bouchitte1988integral}.
\begin{lemma}\label{lem:dualitymeasures}
Let $X$ be a compact metric space and $f:X\times\R^n \to \R\cup \{\infty\}$ a l.s.c.\ function such that for all $x\in X$, $f_x(\cdot)=f(x,\cdot)$ is convex, positively $1$-homogeneous and proper. Then $I_f: \mathcal{M}(X)^n\to\R\cup \{\infty\}$ and $I_{f^*}:\mathcal{C}(X)^n\to\R\cup \{\infty\}$ defined as
\[
I_{f}(\mu) \eqdef \int_X f_x(\dens{\mu}{\lambda})\d \lambda 
\quad \text{and}\quad
 I_{f^*}(\phi) \eqdef \begin{cases}
0 &\text{if $\phi(x)\in \mathrm{dom} f(x,\cdot)^*$, $\forall x \in X$},\\
\infty &\text{otherwise}
\end{cases}
\]
form a pair of convex, proper, l.s.c.\ conjugates functions, where the topology considered are the sup-norm topology for $\mathcal{C}(X)^n$ and the weak* topology for $\mathcal{M}(X)^n$. In the definition of $I_f$, $\lambda$ is any nonnegative measure that dominates $\mu$.
\end{lemma}

An important property of these dynamic unbalanced problems is that in many cases they define geodesic distances on the space of nonnegative measures.
\begin{proposition}[Metric property]\label{prop:metricproperty}
Let $f$ be an infinitesimal cost such that $f(x,1,0,1)$ is bounded on $\Omega$. If for all $x\in \Omega$, the function $L_x:(v,g)\mapsto f(x,1,v,g)$ is positively $p$-homogeneous (for some $p>1$) and symmetric with respect to the origin then $C_D^{1/p}$ is a metric and $(\mathcal{M}_+(\Omega),C_D^{1/p})$ is a geodesic metric space.
\end{proposition}
\begin{proof}
Let $\rho_0,\rho_1\in \mathcal{M}_+(\Omega)$. It is clear that $C_D(\rho_0,\rho_1)$ is finite by Proposition \ref{prop: finite cost dynamic}. The symmetry property comes from the symmetry of $L_x$ and the fact that $(\rho,\omega,\zeta)\in \CE_0^1(\rho_0,\rho_1) \Leftrightarrow (\rho,-\omega,-\zeta)\in \CE_0^1(\rho_1,\rho_0)$. It is clear that $C_D(\rho_0,\rho_0)=0$ and conversely, if $C_D(\rho_0,\rho_1)=0$, then $\omega=\zeta=0$ which implies $\rho_0=\rho_1$. Finally, the triangle inequality follows from equation~\eqref{eq:dynalternative} in Lemma~\ref{lem:constantspeed} below, which also proves that any pair of points can be joined by a constant speed minimizing geodesic.
\end{proof}

\begin{lemma}[Constant speed minimizers]\label{lem:constantspeed}
Let $f$ be an infinitesimal cost such that $f(x,1,0,1)$ is bounded on $\Omega$. If for all $x\in \Omega$, the function $L_x:(v,g)\mapsto f(x,1,v,g)$ is positively $p$-homogeneous (for some $p>1$), then minimizers $(\rho,\omega,\zeta)$ of the dynamic problem of Definition~\ref{def:DynamicProblem} can be disintegrated in time w.r.t.\ Lebesgue and satisfy $C_D(\rho_s,\rho_t)=|t-s|C_D(\rho_0,\rho_1)$ for all $0\leq s<t\leq 1$. Moreover, one has for any $T>0$
\begin{equation}\label{eq:dynalternative}
C_D(\rho_0,\rho_1)^{1/p} = \inf \left\{ \int_0^T \left(\int_{\Omega} f(x,\dens{\rho_t}{\lambda_t},\dens{\omega_t}{\lambda_t},\dens{\zeta_t}{\lambda_t})\d \lambda_t \right)^{1/p}\d t\right\}
\end{equation}
where the infimum runs over $(\rho_t,\omega_t,\zeta_t)_{t\in [0,1]}\in \CE_0^T(\rho_0,\rho_1)$ and $\lambda_t$ is a dummy reference measure that dominates $\rho_t,\omega_t$ and $\zeta_t$ for all $t\in [0,T]$.
\end{lemma}
\begin{proof}
By Proposition~ \ref{prop: finite cost dynamic}, we know that $C_D(\rho_0,\rho_1)$ is always finite. Moreover, since $L_x$ is superlinear, any feasible triplet $(\rho,\omega,\zeta)$ satisfies $\omega, \zeta \ll \rho$ from which we can deduce that $\rho$ admits a disintegration in time with respect to the Lebesgue measure on $[0,1]$ (this comes from the fact that $\zeta_t(\Omega)$ is the distributional derivative of $t\mapsto \rho_t(\Omega)$). Let us denote by $\tilde C$ the infimum in~\eqref{eq:dynalternative}, taken with $T=1$ (the fact that this value does not change with $T$ is a consequence of a simple rescaling argument). One may argue exactly as in~\cite[Thm. 5.4]{dolbeault2009new} to show the inequality $C_D\leq\tilde C$. The reverse inequality follows from H\"older inequality and is exact if and only if the integrand equals $C_D(\rho_0,\rho_1)$ $\d t$-a.e.\ for any minimizer $(\rho_t,\omega_t,\zeta_t)_{t\in [0,1]}$. This constant speed property, combined with the fact that $(\rho,\omega,\zeta)$, after time rescaling, remains minimizing between any pair of intermediate times $0\leq s<t\leq1$ (otherwise one could improve the action on $[0,1]$ by glueing), leads to the constant speed property.
\end{proof}

The main goal of this article is to formulate a Kantorovich-like formulation for problems of this kind. This aim requires to introduce a new definition of \emph{static} optimal transport problems, and this is the object of the next section. The relationship between these two classes of models is stated and proved in Section~\ref{sec:DynamicToStatic}.


\section{Static Kantorovich Formulations}
\label{sec : general static problem}
The classical Kantorovich formulation of optimal transport, for two probability measures $\mu,\nu$ on a set $X$ and with a transport cost $c:(x,y)\mapsto \R \cup \{\infty\}$, is
\[
\inf \left\{ \int_{X^2} c(x,y) \d \gamma(x,y)\;;\; \gamma \in \mathcal{M}_+(X^2), (\Proj_0)_\# \gamma = \mu \text{ and } (\Proj_1)_\# \gamma = \nu \right\}.
\]
A nonnegative measure on $X^2$ that satisfies the marginal constraints is called a \emph{coupling} of $\mu$ and $\nu$. If now $\mu$ and $\nu$ have different total masses, two issues arise: (i) the set of couplings between $\mu$ and $\nu$ is empty and (ii) the cost needs also to describe the effort necessary to make the mass vary. We propose in this section a generalized Kantorovich problem that is valid for arbitrary nonnegative measures. 

\subsection{Definitions}
In what follows, $\Omega \subset \R^d$ is a compact set, $x$ typically refers to a point in $X$ and $m$ to a mass.  We first define a cost function, which takes as input not only two points in space but also two masses: it can be though of as the cost of matching two Dirac measures of arbitrary mass.

\begin{definition}[Cost function]
	\label{def:CostFunction}
	In the sequel, a \emph{cost function} is a function 
	\[
		c : 
		\begin{array}{lcl}
			(\Omega \times [0, + \infty[) ^2 &\to & [0,+ \infty] \\
			(x_0,m_0),(x_1,m_1) &\mapsto &c(x_0,m_0,x_1,m_1)
		\end{array}
	\]
which is l.s.c.\ in all its arguments and jointly sublinear in $(m_0, m_1)$. It is implicitly defined as $+\infty$ outside of its domain of definition.
\end{definition}

A \emph{sublinear} function is by definition a positively 1-homogeneous and subadditive function, or equivalently, a positively 1-homogeneous and convex function. The joint subadditivity of $c$ in $(m_0,m_1)$ guarantees that it is always better to send mass from one point to another in one single chunk. This sublinearity requirement can be interpreted as the generalization of the fact that, in classical optimal transport, the cost of moving mass is linear with the mass.

In order to allow for variations of mass, we need to adapt the constraint set of standard optimal transport by introducing the notion of \emph{semi-couplings}. These are relaxed couplings with only one marginal being fixed. 

\begin{definition}[Semi-couplings]
For two marginals $\rho_0, \rho_1 \in \mathcal{M}_+(\Omega)$, the set of semi-couplings is
	\begin{align}
		\Gamma(\rho_0,\rho_1) \eqdef \left\{
			(\gamma_0,\gamma_1) \in \big(\mathcal{M}_+(\Omega^2)\big)^2 \colon
			\pfwd{(\Proj_0)} \gamma_0 = \rho_0,\, \pfwd{(\Proj_1)} \gamma_1 = \rho_1
			\right\}.
	\end{align}
\end{definition}

Informally, $\gamma_0(x,y)$ represents the amount of mass that is taken from $\rho_0$ at point $x$ and is then transported to an (possibly different, to account for creation/destruction) amount of mass $\gamma_1(x,y)$ at point $y$ of $\rho_1$. These semi-couplings allow us to formulate a novel static Kantorovich formulation of unbalanced optimal transport as follows. 

\begin{definition}[Unbalanced Kantorovich problem]
	\label{def:extended kantorovich}
	For a cost function $c$ we introduce the functional
	\begin{align}
		J_K(\gamma_0,\gamma_1) \eqdef \int_{\Omega^2} c \left( x,\dens{ \gamma_0}{ \gamma},y,\dens{ \gamma_1}{ \gamma}\right) \d \gamma(x,y)\, ,
	\end{align}
where $\gamma \in \mathcal{M}_+(\Omega^2)$ is any measure such that $\gamma_0, \gamma_1 \ll \gamma$. This functional is well-defined since $c$ is jointly 1-homogeneous w.r.t.\ the mass variables (see Definition \ref{def:CostFunction}).
	The corresponding optimization problem is
	\begin{align}
	\label{eq: static problem}
		C_K(\rho_0,\rho_1)  \eqdef \inf_{(\gamma_0,\gamma_1) \in \Gamma(\rho_0,\rho_1)}
			J_K(\gamma_0,\gamma_1)\,.
	\end{align}
\end{definition}

\begin{proposition}
	\label{prop:KMinimizers}
	If $c$ is a \emph{cost function} then a minimizer for $C_K(\rho_0,\rho_1)$ exists.
\end{proposition}
\begin{proof}
By Lemma~\ref{lem:dualitymeasures}, $J_K$ is weakly* l.s.c.\ on $\mathcal{M}(\Omega^2)$. Since $\Omega$ is compact and the marginals $\rho_0,\rho_1$ have finite mass, $\Gamma(\rho_0,\rho_1)$ is tight and thus weakly* pre-compact. It is also closed so $\Gamma(\rho_0,\rho_1)$ is weakly* compact and any minimizing sequence admits a cluster point which is a minimizer (the minimum is not assumed to be finite).
\end{proof}

\begin{example}
\label{ex : example cost static}
	Standard optimal transport problems with a nonnegative, l.s.c.\ cost $\tilde{c}$ are retrieved as particular cases, by taking 
	\begin{align*}
		c(x_0,m_0,x_1,m_1) & = 
		\begin{cases}
			m_0 \cdot \tilde{c}(x,y)& \text{if } m_0 = m_1\,, \\
			+ \infty & \text{otherwise.}
		\end{cases}
	\end{align*}
	\end{example}

This example shows that, unlike the dynamic problems of the previous section, these \emph{semi-coupling} formulations form a proper generalization of the Kantorovich problems. In particular, the properties that we prove next (duality, metric) can be particularized to recover well-known properties of classical optimal transport. Other examples, in particular the $\WF$-metric, are studied in more detail in Section \ref{sec:Examples}.


%

\subsection{Properties of semi-coupling problems}

A central property of optimal transport is that it can be used to lift a metric from the base space $\Omega$ to the space of probability measures over $\Omega$ (cf.\ \cite[Chapter 6]{villani2009oldnew}). We now show that this extends to the unbalanced framework. We first introduce the space $\Cone(\Omega)$, a standard construction in topology which, in our context, can be understood as the space of Dirac measures of arbitrary mass (endowed with the weak* topology).

\begin{definition}[Cone]
The space $\Cone(\Omega)$ is defined as the space $\Omega \times \R_+$ where all the points with zero mass $\Omega \times \{0\}$ are identified to one point. It is endowed with the quotient topology (note that subtleties appear in the non-compact setting, see~\cite{LieroMielkeSavareLong}).
\end{definition}
\begin{theorem}[Metric]
\label{th:metric}
Let $c$ be a cost function such that, for some $p \in [1, + \infty[$ 
\begin{align}
	\label{eq:CMetric}
	(x_0,m_0), (x_1,m_1) \mapsto c(x_0,m_0,x_1,m_1)^{1/p}
\end{align}
is a metric on  $\Cone(\Omega)$. Then $C_K^{1/p}$ defines a metric on $\mathcal{M}_+(\Omega)$.
\end{theorem}
\begin{remark}
Compatibility with the cone structure implies in particular that for all $x_2\in \Omega$, $m\geq0$, 
$c(x_1, 0, x_2, m)$ must be independent of $x_1\in \Omega$ and $c(x_1,0,x_2,0)=0$.
\end{remark}
\begin{remark}
One can replace the word ``metric'' by ``extended metric'' (i.e.\ allowing the value $+\infty$) and the proof goes through. By considering the cost in Example \ref{ex : example cost static}, a corollary is that $C_K^{1/p}$ defines a proper metric on each equivalence class for the relation $\mu \sim \nu \Leftrightarrow \mu(\Omega) = \nu(\Omega)$. The metric property of the Wasserstein distance is thus recovered as a particular case.
\end{remark}

\begin{proof}
First, symmetry and nonnegativity are inherited from $c$. Moreover, 
\begin{multline*}
[ C_K(\gamma_0,\gamma_1) = 0 ]
\Leftrightarrow \\
[ \exists (\gamma_0, \gamma_1) \in \Gamma(\rho_0, \rho_1) : (\gamma_0= \gamma_1) \tn{ and }  (x=y \tn{ $\gamma_0$-a.e.}) ] \\
\Leftrightarrow
[ \rho_0 = \rho_1 ].
\end{multline*}

\newcommand{\mudens}[1]{{\textstyle{\frac{#1}{\mu}}}}
\newcommand{\mudisint}[1]{{(#1|\mu)}}
It remains to show the triangle inequality. Fix $\rho_0$, $\rho_1$, $\rho_2 \in \mathcal{M}_+(\Omega)$. Take two pairs of minimizers for \eqref{eq: static problem} 
	\begin{align*}
		(\gamma_0^{01},\gamma_1^{01}) & \in \Gamma(\rho_0,\rho_1)\,, &
		(\gamma_0^{12},\gamma_1^{12}) & \in \Gamma(\rho_1,\rho_2)\,, 		
	\end{align*}
and let $\mu \in \mathcal{M}_+(\Omega)$ be such that $\pfwd{(\Proj_1)} (\gamma_0^{01},\gamma_1^{01}) \ll \mu$ and $\pfwd{(\Proj_0)} (\gamma_0^{12},\gamma_1^{12}) \ll \mu $. Denote by $\mudisint{\gamma_i^{01}}(x|y)$ the disintegration of $\gamma_i^{01}$ along the second factor w.r.t.\ $\mu$. More precisely, for all $y\in \Omega$, $\mudisint{\gamma_i^{01}}(\cdot|y)\in \mathcal{M}_+(\Omega)$ and it holds, for all $f$ measurable on $\Omega^2$, 
\begin{align*}
	\int_{\Omega^2} f\,\d\gamma_i^{01} & = \int_{\Omega} \left( \int_{\Omega} f(x,y)\,\d\mudisint{\gamma_i^{01}}(x|y) \right)\,\d\mu(y)
\end{align*}
and analogously for $\mudisint{\gamma_i^{12}}(z|y)$ along the first factor for $i=0,1$. Write $\mudens{\rho_1}(y)$ for the density of $\rho_1$ w.r.t.\ $\mu$.
%
We combine the optimal semi-couplings in a suitable way to define $\gamma_0$, $\gamma_1$, $\hat{\gamma} \in \mc{M}(\Omega^3)$ (via disintegration w.r.t.\ $\mu$ along the second factor):
\begin{align*}
\mudisint{\gamma_0}(x,z|y) &\eqdef
\begin{cases}
\frac{\mudisint{\gamma_0^{01}}(x|y) \otimes \mudisint{\gamma_0^{12}}(z|y)}{\mudens{\rho_1}(y)} & \tn{ if } \mudens{\rho_1}(y) >0, \\
\mudisint{\gamma_0^{01}}(x|y) \otimes \delta_{y}(z) & \tn{ otherwise, }
\end{cases} \\
\mudisint{\gamma_1}(x,z|y) &\eqdef
\begin{cases}
\frac{\mudisint{\gamma_1^{01}}(x|y) \otimes \mudisint{\gamma_1^{12}}(z|y)}{\mudens{\rho_1}(y)} & \tn{ if } \mudens{\rho_1}(y) >0, \\
\delta_{y}(x) \otimes \mudisint{\gamma_1^{12}}(z|y) & \tn{ otherwise, }
\end{cases} \\
\mudisint{\hat{\gamma}}(x,z|y) &\eqdef  
\begin{cases}
\frac{\mudisint{\gamma_1^{01}}(x|y) \otimes \mudisint{\gamma_0^{12}}(z|y)}{\mudens{\rho_1}(y)} & \tn{ if } \mudens{\rho_1}(y) >0, \\
0 & \tn{ otherwise. }
\end{cases}
\end{align*}
The interpretation of $\gamma_0$ is that all mass that leaves $x$ towards $y$, according to $\gamma_0^{01}(x,y)$, is distributed over the third factor according to $\gamma_0^{12}(y,z)$. In case the mass disappears at $y$, it is simply ``dropped'' as $\delta_y$ on the third factor. Then $\gamma_1$ is built analogously for the incoming masses and $\hat{\gamma}$ is a combination of incoming and outgoing masses.
For $i=0,1$ let $\gamma_i^{02} \eqdef \pfwd{(\Proj_{02})} \gamma_i$ and note that, by construction, $(\gamma_0^{02},\gamma_1^{02}) \in \Gamma(\rho_0, \rho_2)$.
In the rest of the proof, for an improved readability, when writing the functional the dummy measure $\gamma$ such that $\gamma_0, \gamma_1 \ll \gamma$ is considered as implicit and we write
	\begin{equation*}
	\label{remark : abuse of notation}
		\int_{\Omega^2} c\big(x,\gamma_0(x,y),y,\gamma_1(x,y)\big)\,\d x\,\d y
		\quad \text{for}\quad \int_{\Omega^2} c \left(x,\dens{ \gamma_0}{ \gamma},y,\dens{ \gamma_1}{ \gamma}\right) \d \gamma(x,y)\, .
	\end{equation*}
With this notation, one has
\begin{align*}
& \int_{\Omega^3} c\big(x,\gamma_0(x,y,z),y,\hat{\gamma}(x,y,z)\big) \d x\,\d y\,\d z \\
= {} & \int_{\rho(y)>0} \left( \int_{\Omega^2}
	c\big(x,\mudisint{\gamma_0^{01}}(x|y), y, \mudisint{\gamma_1^{01}}(x|y)\big)
	\frac{\mudisint{\gamma_0^{12}}(y|z)}{\mudens{\rho}(y)} \d x\,\d z \right) \d \mu(y) \\
& + \int_{\rho(y)=0} \left( \int_{\Omega^2}
	c\big(x,\mudisint{\gamma_0^{01}}(x|y),y,0\big) \delta_{y}(z) \d x\,\d z \right) \d \mu(y) \\
= {} &  \int_{\Omega} \left( \int_{\Omega}
	c\big(x,\mudisint{\gamma_0^{01}}(x|y), y, \mudisint{\gamma_1^{01}}(x|y)\big)
	\d x \right) \d\mu(y) \\
= {} & J_K(\gamma_0^{01},\gamma_1^{01}) = C_K(\rho_0,\rho_1)\, ,
\end{align*}
and analogously
\begin{align*}
	\int_{\Omega^3} c\big(y,\hat{\gamma}(x,y,z),z,\gamma_1(x,y,z)\big) \d x\,\d y\,\d z & =
	C_K(\rho_1,\rho_2)\,.
\end{align*}
One finally obtains
	\begin{align*}
		 C_K(\rho_0,\rho_2)^{\frac{1}{p}} 
		 \leq {} & \Big(\int_{\Omega^2} c\big(x,\gamma_0^{02}(x,z),z,\gamma_1^{02}(x,z) \big) \d x\,\d z \Big)^{\frac{1}{p}} \\
		\overset{(1)}{\leq} {} & \Big( \int_{\Omega^3} c\big( x, \gamma_0(x,y,z),z, \gamma_1(x,y,z) \big) \d x\,\d y\,\d z \Big)^{\frac{1}{p}} \\
		\overset{(2)}{\leq} {} & 
			\Big( \int_{\Omega^3} 
			\Big[ 
			c\big( x, \gamma_0(x,y,z), y, \hat{\gamma}(x,y,z) \big)^{\frac{1}{p}} + \\
			&\qquad\qquad c\big( y, \hat{\gamma}(x,y,z), z, \gamma_1(x,y,z) \big)^{\frac{1}{p}}
			\Big]^p \d x\,\d y\,\d z 
			\Big)^{\frac{1}{p}} \\
		\overset{(3}{\leq} {} & 
			\Big( \int_{\Omega^3} c\big( x, \gamma_0(x,y,z),y, \hat{\gamma}(x,y,z)\big) \d x\,\d y\,\d z \Big)^{\frac{1}{p}} + \\
			&\qquad\qquad  \Big( \int_{\Omega^3} c\big( y, \hat{\gamma}(x,y,z), z,\gamma_1(x,y,z) \big) \d x\,\d y\,\d z\Big)^{\frac{1}{p}} \\
		\overset{(4)}{=} {} & C_K\left(\rho_0, \rho_1\right)^{\frac{1}{p}} + 
			C_K\left(\rho_1, \rho_2\right)^{\frac{1}{p}} 
		\end{align*}
	where we used (1) the convexity of $c$, (2) the fact that $c^{1/p}$ satisfies the triangle inequality, (3) Minkowski's inequality and (4) comes from the computations above. Thus $C_K(\cdot,\cdot)^{1/p}$ satisfies the triangle inequality, and is a metric. 
\end{proof}

Let us now give a dual formulation of the problem. It is similar to the dual formulation of the Kantorovich problem where one maximizes $\int \phi \d \rho_0 + \int \psi \d \rho_1$ over the set of pairs of continuous functions $(\phi,\psi)$ on $\Omega$ that satisfy $\phi(x)+\psi(y) \leq c(x,y)$. Here the objective is the same, but the constraint set is generally not a linear constraint set and is modified as follows.

\begin{definition}[Dual contraint set]\label{def:dualconstraint}
The dual constraint set $Q:\Omega^2 \to 2^{\R\times\R}$ is a set valued function defined, for all $(x,y)\in \Omega^2$ as the polar set to the function $c(x,\cdot,y\cdot)$, i.e.\ the domain of its convex conjugate.
\end{definition}

Remember that since the cost function $c$ is jointly 1-homogeneous, convex, and l.s.c.\ in the variables $(m_0,m_1)$, for all $(x,y)\in \Omega^2$, the Legendre transform of $c(x,\cdot,y,\cdot)$ is the indicator of a closed convex set. Moreover, as $c$ is nonnegative, and worth $+\infty$ if $m_0$ or $m_1$ is negative, the latter contains the negative orthant. We now state the dual problem (note that we do not attempt to prove the existence of dual maximizers).


\begin{proposition}[Duality]
\label{prop: duality}
Let $c$ be a \emph{cost function}, $Q$ the associated dual constraint set from Definition~\ref{def:dualconstraint} and let
\[ 
B =  \left\{ (\phi, \psi) \in C(\Omega)^2 :  \forall (x,y) \in \Omega^2, (\phi(x),\psi(y)) \in Q(x,y) \right\}\,.
\]
Then
\[
C_K(\rho_0,\rho_1) =
 \sup_{(\phi, \psi) \in B} \int_{\Omega} \phi(x) \d \rho_0 + \int_{\Omega} \psi(y) \d \rho_1 
\]
\end{proposition}

\begin{corollary}[Sublinearity]\label{cor:sublinearityCK}
As the supremum of continuous linear functionals, $C_K$ is a sublinear (i.e.\ convex, positively $1$-homogeneous) and weakly* lower semicontinuous functional on $\mathcal{M}(\Omega)^2$.
\end{corollary}

\begin{proof}
Let us rewrite  the supremum problem as
\[
\sup_{(u_0,u_1) \in (C(\Omega^2)^2} - F(u_0,u_1) - G(u_0,u_1)
\]
where
\begin{align*}
G &: (u_0,u_1) \mapsto 
\begin{cases}
- \int_{\Omega}  \phi(x) \d \rho_0 -  \int_{\Omega} \psi(y) \d \rho_1  & \tn{if } u_0(x,y)=\phi(x) \\
		& \tn{and } u_1(x,y) = \psi(y) \\
+ \infty &\tn{otherwise,}
\end{cases} 
\end{align*}
and $F$ is the indicator of $\{ (u_0,u_1)\in (C(\Omega^2)^2) : (u_0,u_1)(x,y) \in Q(x,y), \, \forall (x,y)\in \Omega^2 \}$.
Note that $F$ and $G$ are convex and proper. Also, given our assumptions, there is a pair of functions $(u_0,u_1)$ at which $F$ is continuous (for the $\sup$ norm topology) and $F$ and $G$ are finite since for all $(x,y)\in \Omega^2$, $Q(x,y)$ contains the negative orthant. Then Fenchel-Rockafellar duality theorem (see, e.g.\ \cite[Theorem 1.9]{cedric2003topics}) states that
 \begin{equation}
 \label{eq: Fenchel-Rockafellar}
 \sup_{(u_0,u_1) \in C(\Omega^2)^2} \!\!\!\!\!\! - F(u_0,u_1) - G(u_0,u_1) 
 = 
 \!\!\!\!\!\!\min_{(\gamma_0, \gamma_1) \in \mathcal{M}(\Omega^2)^2} \!\!\!\!\! \left\{ F^*(\gamma_0,\gamma_1) + G^*(-\gamma_0,-\gamma_1) \right\} .
 \end{equation}
Let us compute the Legendre transforms. For $G$, we obtain
\begin{align*}
G^*(-\gamma_0,-\gamma_1) 
&= \!\!\!\!\!\! \sup_{(\phi,\psi)\in C(\Omega)^2} \!\!\!\! - \int_{\Omega^2} \phi(x)\d \gamma_0 - \int_{\Omega^2} \psi(x)\d \gamma_1 + \int_{\Omega} \phi(x) \d \rho_0 + \int_{\Omega} \psi(y) \d \rho_1\\
&= 
\begin{cases}
0 & \tn{ if } (\gamma_0, \gamma_1)\in \Gamma^{+/-}_{\rho_0, \rho_1}\\
+ \infty & \tn{ otherwise.}
\end{cases}
\end{align*}
where $\Gamma^{+/-}_{\rho_0, \rho_1} $ is the set of semi-couplings without the non-negativity constraint. On the other hand, by Lemma~\ref{lem:dualitymeasures},
\[
F^*(\gamma_0,\gamma_1) = 
\int_{\Omega^2} c\left(x, \dens{ \gamma_0}{ \gamma},y, \dens{ \gamma_1}{ \gamma} \right) \d\gamma(x,y) 
\]
where $\gamma$ is any measure in $\mathcal{M}_+(\Omega^2)$ with respect to which $(\gamma_0,\gamma_1)$ is absolutely continuous. 
Finally, as $F^*$ includes the nonnegativity constraint, the right hand side of \eqref{eq: Fenchel-Rockafellar} is equal to $\min_{(\gamma_0,\gamma_1)\in \Gamma(\rho_0,\rho_1)} J_K(\gamma_0,\gamma_1)$.
\end{proof}

\section{From Dynamic to Static Problems}
\label{sec:DynamicToStatic}
In the previous sections, we have introduced and studied basic properties of two categories of unbalanced optimal transport problems: dynamic formulations and semi-coupling formulations. In this section, we prove (Theorem~\ref{th: continuous static general}) that any dynamic problem is equivalent to a semi-coupling problem. We put ourselves back in the setting of Section~\ref{sec:dynamic} where $\Omega$ is the closure of an open, connected, bounded subset of $\R^d$ with Lipschitz boundary.

\subsection{Minimal path cost and convexification}

Any infinitesimal cost $f$ defines a cost function on $\Omega \times \R_+$ obtained by minimizing over absolutely continuous trajectories.

\begin{definition}[minimal path cost]
The minimal path cost associated to the Lagrangian $L$ is, for $(x_i,m_i)\in \Omega\times\R_+$ with $i\in \{0,1\}$,
\begin{equation}\label{eq-pathspace}
c_f(x_0,m_0,x_1,m_1) \eqdef \inf_{ x(t),m(t)}  \int_0^1 f(x(t),m(t),m(t)\, x'(t),m'(t)) \, \d t\,
\end{equation}
where the infimum runs over absolutely continuous paths $t\mapsto(x(t),m(t))$ that link $(x_0,m_0)$ to $(x_1,m_1)$.
\end{definition}

In general, $c_f$ does not define a cost function in the sense of Definition \ref{def:CostFunction} because of a possible lack of convexity, but its convex regularization does so. This convex relaxation can be nicely expressed as the minimal path cost obtained by allowing mass to be split in two ``travelling'' Dirac masses.

\begin{proposition}[properties and convexification]\label{prop : alternative dynamic/static}
Assume that properties C\ref{ass: multiplicative} and C\ref{ass: doubling} hold for the infinitesimal cost $f$. Then the minimal path cost $c_f$ is finite and continuous on $(\Omega\times\R_+)^2$. Morevoer, its convex regularization $\tilde c_f$ is an admissible cost function, is continuous, and characterized by
\begin{equation}\label{eq:c one or two diracs}
\tilde c_f(x_0,m_0,x_1,m_1 ) = \min_{\substack{m_0^a+m_0^b=m_0\\m_1^a+m_1^b=m_1}} c_f(x_0,m_0^a,x_1,m_1^a)+ c_f(x_0,m_0^b,x_1,m_1^b)\, .
\end{equation}
\end{proposition}
\begin{proof}
The argument of the proof of Proposition~\ref{prop: finite cost dynamic} shows that $c_f$ is finite on its domain. It is also continuous since one can perturb a path and reach nearby points with a perturbed cost (thanks to assumption C\ref{ass: multiplicative}). As for $\tilde c_f$, it is clear that it takes values in $[0,\infty]$ and that it inherits $1$-homogeneity from $f$, so $\tilde c_f(x,\cdot,y\cdot)$ is sublinear for all $(x,y)$. Moreover, characterization is a consequence of Carath\'eodory's Theorem \cite[Corollary 17.1.6]{rockafellar2015convex} (one can also check directly that is it valid for $(m_0,m_1)=(0,0)$ although this is not part of the cited result). The infimum is attained because $c_f$ is continuous and the minimization set is compact. Finally, the continuity of $\tilde c_f$ is clear from this characterization and the continuity of $c_f$.
\end{proof}
%

\subsection{Main Theorem}

We can now state our main theorem, which connects the dynamic and the semi-coupling formulations of unbalanced optimal transport.

\begin{theorem}
\label{th: continuous static general}
Let $\Omega \subset \R^d$ be a compact which is star shaped w.r.t.\ a set of points with nonempty interior, and $f$ be an infinitesimal cost satisfying (C\ref{ass: multiplicative})  and (C\ref{ass: doubling}). Defining $\tilde c_f$ the convex relaxation of the minimal path cost (as in Proposition~\ref{prop : alternative dynamic/static}) as the static cost, it holds, for $\rho_0, \rho_1 \in \mathcal{M}_+(\Omega)$, 
\[
C_D(\rho_0, \rho_1) = C_K(\rho_0,\rho_1).
\]
\end{theorem}
%
%
The assumption on the domain includes convex sets but also most star-shaped sets. This particular assumption comes from the fact that our proof strategy involves a smoothing step that needs to deal correctly with the dependency in $x$ of the infinitesimal cost $f$.

\begin{proof}[Proof of Theorem \ref{th: continuous static general}]
The proof is divided into three steps: in Step 1, we show, using discrete approximations of semi-couplings, that it holds $C_K\geq C_D$. By integrating characteristics (an argument similar to the original proof of the Benamou-Brenier formula \cite{benamou2000computational}), we show in Step 2 that for absolutely continuous marginals, $C_K$ is upper bounded by the dynamic minimization problem restricted to smooth fields. In Step 3, a regularization argument, inspired by \cite[Theorem 8.1]{cedric2003topics}, extends this result to general measures and shows that $C_K\leq C_D$. 

\paragraph{Step 1.}
Let $(\rho_0,\rho_1)\in \mathcal{M}_+(\Omega)^2$ and 
let $(\hat \gamma_0^{(n)})$, let $(\hat \gamma_1^{(n)})$ be the atomic measures on $\Omega^2$ given by Lemma \ref{lem:approximationsemicoupling} below (the continuity of $h$ follows from Proposition \ref{prop : alternative dynamic/static}): they are such that  $\lim_{n\to \infty} J_K(\hat \gamma_0^{(n)},\hat \gamma_1^{(n)}) = C_K(\rho_0,\rho_1)$. Also denote by $(\rho_0^n,\rho_1^n)\eqdef((\Proj_0)_\# \bar \gamma^{(n)}_0,\allowbreak (\Proj_1)_\# \bar \gamma^{(n)}_1)$ the corresponding marginals, that satisfy $\rho_0^n\rightharpoonup^* \rho_0$ and $\rho_1^n\rightharpoonup \rho_1$. 

From the definition of $c_f$ and by characterization \eqref{eq:c one or two diracs} of $\tilde c_f$ one has that $\tilde c_f(x_0,m_0,x_1,m_1)\geq C_D(m_0\delta_{x_0},m_1\delta_{x_1})$. It follows, taking the notations of the proof of Lemma \ref{lem:approximationsemicoupling},
\begin{align*}
J_K(\hat \gamma_0^{(n)},\hat \gamma_1^{(n)})
&= \sum_{i\in I} \tilde c_f(x_i, f^{(n)}_i,y_i,g^{(n)}_i)\gamma_i^{(n)}\\
&\geq \sum_{i\in I} C_D(f^{(n)}_i  \gamma_i^{(n)} \delta_{x_i}, g^{(n)}_i  \gamma_i^{(n)} \delta_{y_i}) \geq C_D(\rho_0^n,\rho_1^n).
\end{align*}
Since by Corollary~\ref{cor:sublinearityCD}, $C_D$ is weakly l.s.c., it follows $C_D(\mu,\nu)\leq C_K(\mu,\nu)$.

\paragraph{Step 2.} Let $\rho_0,\, \rho_1 \in  \mathcal{M}_+(\Omega)$ be absolutely continuous measures with positive mass and let
\[
(\rho,\omega,\zeta) \in \left\{ (\rho,\omega,\zeta)\in \mathcal{CE}_0^1(\rho_0,\rho_1) : \dens{\omega}{\rho}, \dens{\zeta}{\rho}\in C^1([0,1]\times \Omega) \right\} \, .
\]
%
%
One can take the Lagrangian coordinates $(\varphi_t(x),\lambda_t(x))$ which are given by the flow of the smooth fields $(v\eqdef \omega/\rho, g \eqdef \zeta/\rho)$ (see \cite[Prop. 3.6]{maniglia2007probabilistic}) :
\begin{align*}
\partial_t \varphi_t(x) = v_t(\varphi_t(x)) \quad \tn{and} \quad
\partial_t  \lambda_t(x) = g_t(\varphi_t(x))\lambda_t(x) \, ,
\end{align*}
with the initial condition $(\varphi_0(x), \lambda_0(x)) = (x,1)$. The Neumann boundary conditions on $v$, included in the continuity equation, imply that for all $t\in [0,1]$, $\Omega$ is stable by $\varphi_t$. Recall that $(\varphi_t(x),\lambda_t(x))$ describes the position and the relative increase of mass at time $t$ of a particle initially at position $x$ and that one has $\rho_t = (\varphi_t)_\# (\lambda_t \cdot \rho_0)$.
It follows
\begin{align*}
J_D(\rho,\omega, \zeta) 
&= \int_{[0,1] \times \Omega} f\big(x,1,v_t(x),g_t(x)\big) \d[(\varphi_t)_{*} \big(\lambda_t \rho_0 \big)](x) \d t \\
&\overset{(1)}{=} \int_{\Omega} \left[ \int_0^1 f\big(\varphi_t(x),1, \partial_t \varphi_t(x), \partial_t \lambda_t(x)/\lambda_t(x) \big) \lambda_t(x) \d t \right] \d \rho_0(x) \\
&\overset{(2)}{=}  \int_{\Omega} \left[ \int_0^1 f\big(\varphi_t(x),\lambda_t(x), \lambda_t(x) (\partial_t \varphi_t(x)),\partial_t \lambda_t(x) \big) \d t \right] d \rho_0(x) \\
&\overset{(3)}{\geq} \int_{\Omega} c_f(x,1, \varphi_1(x), \lambda_1(x)) d \rho_0(x) \\
&\overset{(4)}{\geq}  C_K(\rho_0,\rho_1)
\end{align*}
where we used (1) the change of variables formula (2) homogeneity of $f$ (3) the definition of the minimal path cost $c_f$ and (4) 
the fact that  $\big( \pfwd{(\id, \varphi_1)} \rho_0, \allowbreak \pfwd{(\id, \varphi_1)} (\lambda_1 \rho_0)\big)\in \Gamma(\rho_0,\rho_1)$ and $c_f\geq \tilde c_f$.
%
\paragraph{Step 3.} 
Let $\rho_0, \rho_1 \in \mathcal{M}_+(\Omega)$. We want to show, with the help of Step 2, that $C_K(\rho_0, \rho_1) \leq C_D(\rho_0, \rho_1)$.
Let $(\rho,\omega,\zeta) \in \mathcal{CE}_0^1(\rho_0,\rho_1)$ and for $\delta \in ]0,1[$ let
\[
\tilde{\rho}^{\delta} = (1-\delta) \rho + \delta\, (\d x \otimes \d t)\vert_{S}, \quad \tilde{\omega}^{\delta} = (1-\delta) \omega, \quad \tilde{\zeta}^{\delta} = (1-\delta) \zeta
\]
where $S\supset \Omega$ is a bounded set containing $\Omega+B^d(0,1)$ and $B^d(a,r)$ denotes the open ball of radius $r$ centered at point $a$ in $\R^d$. One has $(\tilde{\rho}^{\delta} ,  \tilde{\omega}^{\delta}, \tilde{\zeta}^{\delta} )\vert_\Omega \in \mathcal{CE}_0^1(\tilde{\rho}_0^{\delta}\vert_\Omega,\tilde{\rho}_1^{\delta}\vert_\Omega)$ and by convexity, 
\[
J_D(\tilde{\rho}^{\delta} ,  \tilde{\omega}^{\delta}, \tilde{\zeta}^{\delta} ) \leq J_D(\rho, \omega, \zeta)\, .
\]
Since $\tilde{\rho}_0^{\delta}\vert_\Omega \rightharpoonup^*\rho_0$ and $\tilde{\rho}_1^{\delta}\vert_\Omega \rightharpoonup^*\rho_1$ as $\delta \to 0$ and $C_K$ is weakly* lower semicontinuous (Corollary~\ref{cor:sublinearityCK}), it is sufficient to prove
$
J_D(\tilde{\rho}^{\delta} ,  \tilde{\omega}^{\delta}, \tilde{\zeta}^{\delta} ) \geq C_K(\tilde{\rho}_0^{\delta}\vert_{\Omega} ,  \tilde{\rho}_1^{\delta}\vert_\Omega ) 
$
for proving $J_D(\rho, \omega, \zeta) \geq C_K(\rho_0,\rho_1)$.
In order to alleviate notations we shall now denote $\tilde{\rho}^{\delta} ,  \tilde{\omega}^{\delta}, \tilde{\zeta}^{\delta} $ by just $\rho, \omega, \zeta$ and the new marginals $\tilde{\rho}_0^{\delta},  \tilde{\rho}_1^{\delta}$ by $\rho_0,\rho_1$ (now supported on $S\supset \Omega$).
Up to a translation, we can assume that $0$ is in the interior of the set of points w.r.t.\ which $\Omega$ is star shaped. Then \cite[Theorem 5.3]{rubinov2013abstract} tells us that the Minkowski gauge $x \mapsto \inf \{ \lambda>0 : x \in \lambda \Omega \}$ is Lipschitz: let us denote by $k\in \R_+^*$ its Lipschitz constant.
We introduce the regularizing kernel
$
r_{\veps}(t,x) = \frac{1}{\veps^{d}}\, r_1\left( \frac{x}{\veps} \right) \frac{1}{\veps}\, r_2 \left( \frac{t}{\veps} \right) \, 
$
where $r_1 \in C_c^{\infty}(B^d(0,\frac{1}{2k}))$, $r_2 \in C_c^{\infty}(B^1(0,\frac{1}{2k}))$, $r_i \geq 0$, $\int r_i = 1$, $r_i$ even ($i=1,2$). We want to perform a smoothing procedure which, in some sense, preserves the domain $\Omega$ because this is where $f$ and $\tilde c_f$ are defined.

Let $\bar{\mu}=((1+\veps)^{-1}\bar{\rho},(1+\veps)^{-1}\omega,\zeta)$ where $\bar{\rho}$ is a measure on $[-\veps,1+\veps]\times S$ which is worth $\rho$ on $[0,1]\times S$,  $\rho_0 \otimes \d t$ on $[-\veps,0[\times S$ and $\rho_1 \otimes \d t$ on $]1,1+\veps]\times S$ (while $\omega, \zeta$ are always implicitely extended by $0$). 
Then define
\[
\mu^\veps \eqdef T_{\#} (\bar{\mu} \ast r_\veps)\vert_{[0,1]\times \Omega},
\]
where $T:(t,x) \mapsto ((1+\veps)^{-1}(t+ \veps/2),(1+\veps)^{-1}x)$ is built in such a way that the image of the time segment $[-\veps/2,1+\veps/2]$ is $[0,1]$. Furthermore, since the Minkowski gauge of $\Omega$ is $k$-Lipschitz, the image of $\Omega^\veps \eqdef \Omega + B^d(0,\veps/k)$ by $T$ is included in $\Omega$. Now, by the smoothing and scaling properties in Proposition \ref{prop:results on continuity equation}, it holds $\mu^{\veps} \in \mathcal{CE}_0^1(\rho^{\veps}_0,\rho^{\veps}_1)$ for some smoothed marginals $(\rho_0^\veps, \rho_1^\veps)$.

Notice that $\rho^{\veps}_0 \rightharpoonup^* \rho_0\vert_\Omega$ and $\rho^{\veps}_1 \rightharpoonup^* \rho_1\vert_\Omega$ when $\veps \to 0$ since these measures are the evaluations of $\bar{\mu}\ast r_\veps$ at time $-\veps/2$ and $1+\veps/2$, respectively (also contracted in space by a factor $1+\veps$ and restricted to $\Omega$). Moreover, the vector fields $\omega_t^{\veps}/ \rho_t^{\veps}$ and $\zeta_t^{\veps}/ \rho_t^{\veps}$ are well-defined, smooth, bounded functions on $[0,1]\times \Omega$ because $\rho^\veps$ has a density bounded from below by a positive constant whenever $\veps/k <1$ (so that $\Omega^\veps\subset S$). Therefore, by Step 2,
\[
J_D(\mu^{\veps}) \geq C_K(\rho_0^{\veps}, \rho_1^{\veps}).
\]
On the other hand, for any $\veps'>0$, one has
\begin{align*}
J_D(\mu^{\veps}) 
& =  \int_0^1 \int_{\Omega} f(y, \dens{\mu^{\veps}}{|\mu^{\veps}|}) \d |\mu^{\veps}|(s,y)  \\
& \overset{(1)}{=}  \int_{-\veps/2}^{1+\veps/2} \int_{\Omega^\veps} f((1-\veps)y, \dens{\bar{\mu} \ast r_{\veps}}{|\bar{\mu} \ast r_{\veps}|}) \d |\bar{\mu}\ast r_{\veps}|(s,y) \\
& \overset{(2)}{\leq}  \int_{-\veps/2}^{1+\veps/2} \int_{\Omega^\veps} \d |\bar{\mu}|(t,x) \int_{\R^{1+d}} \!\!\! \d s \d y f((1-\veps)y, \dens{\bar{\mu}}{|\bar{\mu}|}(t,x)) \cdot r_{\veps}(s-t,y-x)  \\
&  \overset{(3)}{=}  \sum_i \int_{0}^{1} \int_{\Omega} \d |\bar{\mu}|(t,x) \int_{\R^{1+d}} \!\!\! \d s \d y \tilde{f}_i(\dens{\bar{\mu}}{|\bar{\mu}|}(t,x))   \lambda_i((1-\veps)y) r_{\veps}(s-t,y-x)  \\
& \overset{(4)}{\leq} \int_0^{1} \int_{\Omega} (1+\veps')f(x, \dens{\bar{\mu}}{|\bar{\mu}|})  \d |\bar{\mu}|(t,x) 
\end{align*}
where were used (1) the change of variable formula, (2) the convexity and homogeneity of $f$, (3) assumption (C\ref{ass: multiplicative})  (which also garanties the integrability of each $\tilde{f}_i$) assumed on $f$ and (4) the continuity of the strictly positive factors $(\lambda_i)_i$, if for a given $\veps'>0$, one choses $\veps$ small enough. Therefore $ J_D(\bar{\mu}) \geq J_D(\mu^\veps)(1+\veps')^{-1}$. But by convexity and homogeneity, $J_D(\bar{\mu})\leq (1+\veps)^{-1} ((1-\veps)J_D(\rho,\omega,\zeta) + \veps J_D(\rho,\omega,2\zeta))$. By the ``doubling'' assumption (C\ref{ass: doubling})  on $f$, the term $J_D(\rho,\omega,2\zeta)$ is finite if $J_D(\rho,\omega,\zeta)$ is finite and one has
\[
C_K(\rho_0^{\veps}, \rho_1^{\veps}) \leq \frac{1+\veps'}{1+\veps}((1-\veps)J_D(\rho,\omega,\zeta) + \veps J_D(\rho,\omega,2\zeta))\, .
\]
Letting $\veps'$ and $\veps$ go to $0$, using the lower semicontinuity of $C_K$ and taking the infimum, one recovers in the end $C_K(\rho_0\vert_\Omega,\rho_1\vert_\Omega) \leq J_D(\rho,\omega,\zeta)$ as desired.
\end{proof}

\begin{lemma}[Atomic approximation of semi-couplings]\label{lem:approximationsemicoupling}
Let $\mu\in \mathcal{M}_+(\Omega)$ and $\nu\in \mathcal{M}_+(\Omega)$. If $c$ is a continuous cost function, then there exists a sequence of finitely atomic measures $(\gamma_0^{(n)})$ and $(\gamma_1^{(n)})$ that weakly converge to an optimal pair of semi-couplings for $C_K(\mu,\nu)$ and such that 
\[
\lim_{n\to \infty} J_K(\gamma_0^{(n)},\gamma_1^{(n)}) = C_K(\mu,\nu).
\]
\end{lemma}
\begin{proof}
Let $(f \gamma,g \gamma)$ be an optimal pair of semi-couplings for $C_K(\mu,\nu)$, where $ \gamma \in \M_+(\Omega^2)$ and $f,g \in L^1( \gamma)$. Let $(B^{(n)}_i, (x_i,y_i)^{(n)})_{i\in I}$ be sequence of finite pointed partitions of $\Omega^2$ such that $\lim_{n\to \infty}\max_{i\in I} \diam \, B^{(n)}_i=0$.  We define the discrete approximations $\bar \gamma_0^{(n)}\eqdef T^{(n)}_\# (f\gamma)$ and $\bar \gamma_1^{(n)}\eqdef T^{(n)}_\# (g\gamma)$ where $T^{(n)}:\Omega^2\to \Omega^2$ maps all points in $B_i^{(n)}$ to $(x_i,y_i)^{(n)}$.
Also, denote $(\Proj_0)_\# \bar \gamma^{(n)}_0 =\mu_n$ and $(\Proj_1)_\# \bar \gamma^{(n)}_1=\nu_n$.
 It is clear that the discretized semi-couplings weakly converge to $(f\gamma,g\gamma)$. Moreover, for $\epsilon>0$, there exists $n\in \N$ such that for all $i\in I$, by Jensen inequality, and since $c$ is continuous, uniformly on $(X\times [0,1])\times (Y \times [0,1])$,
\begin{multline*}
c(x_i^{(n)},\bar \gamma_0^{(n)}(B_i^{(n)}),y_i^{(n)}, \bar \gamma_1^{(n)}(B_i^{(n)})) \leq \epsilon \max\{\bar \gamma_0^{(n)}(B_i^{(n)}),\bar \gamma_1^{(n)}(B_i^{(n)})\} \\ 
+ \int_{B^{(n)}_i} c(x,f(x,y),y,g(x,y))\d \gamma(x,y)
\end{multline*}
By integrating on the whole domain, one has
\[
J_K(\hat \gamma_0^{(n)},\hat \gamma_1^{(n)}) \leq  \epsilon\cdot (\mu(X)+\nu(Y)) + C_h(\mu,\nu)
\]
and the result follows because $\epsilon$ can be arbitrarily small.
\end{proof}


\section[Examples]{Examples}\label{sec:Examples}

In this section we discuss some examples which fit into the framework developed in sections \ref{sec : general static problem} and \ref{sec:DynamicToStatic}. Optimal partial transport is first treated and then the motivating example: the Wasserstein-Fisher-Rao metric. In this section, $\Omega$ is a convex compact set in $\R^d$ with nonempty interior.

\subsection{Optimal Partial Transport}
\label{sec:ExamplesPartial}

In our first example, we consider infinitesimal costs of the form
\begin{align}
\label{eq:POT dynamic cost}
f : (\rho,\omega,\zeta) \mapsto 
\begin{cases}
\frac1p \frac{|\omega|^p}{\rho^{p-1}} + \delta |\zeta| & \tn{ if } \rho>0 \\
\delta |\zeta|  & \tn{ if } \rho = |\omega| =0 \\
+ \infty & \tn{ otherwise .}
\end{cases}
\end{align}
which satisfies the conditions of Definition \ref{def: infinitesimal cost} and Assumptions (C\ref{ass: multiplicative}-\ref{ass: doubling}). These assumptions also allow for a continuous dependency of $\delta$ in $x$ although we do not consider it here for simplicity.
We first compute the minimal path cost associated to this infinitesimal cost.
\begin{proposition}[Minimal path cost]
Let $(x_0,m_0)$ and $(x_1,m_1)$ be points in $\Omega \times \R_+$. The minimal path cost associated to the infinitesimal cost \eqref{eq:POT dynamic cost} is
\begin{equation}\label{eq: stat/dyn TV}
	c_f (x_0,m_0,x_1,m_1) =
	\min \left( \frac{|x_1-x_0|^p}{p} ,2\delta\right) \cdot \min (m_0,m_1) + \delta |m_1-m_0|.
\end{equation}
It is already sublinear in $(m_0,m_1)$ for all $(x_0,x_1)\in \Omega^2$.
\end{proposition}

\begin{proof}
First, for any absolutely continuous path $(x,m)$ we denote by $\bar{m} = \min_{t\in [0,1]} m(t)$ its minimum mass. It holds
\begin{align*}
\int_0^1 f(m(t),m(t)x'(t),m'(t)) \d t &= \int_0^1 \frac1p |x'(t)|^p m(t) \d t + \delta \int_0^1 |m'(t)| \d t \\
&\geq \bar{m}  \int_0^1 \frac1p |x'(t)|^p\d t + \delta (|m_0-\bar{m} | + |m_1-\bar{m}| ) \\
& \geq c(x_0,m_0,x_1,m_1) \, .
\end{align*}
For the opposite inequality, let us build a minimizing sequence. In the case where $|x_0-x_1|^p/p \leq 2\delta$, we divide the time interval into three segments $[0,\epsilon], [\epsilon, 1-\epsilon]$ and $[1-\epsilon, 1]$ and build an absolutely continuous trajectory by making pure variations of mass (or staying on place) in the first and third segments, and constant speed transport of the mass $\bar{m} = \min(m_0,m_1)$ during the second segment. This way, we obtain that the right-hand side of \eqref{eq: stat/dyn TV} is upper bounded by
\[
|m_1-m_0| + \lim_{\veps \to 0}  \int_\veps^{1-\veps} \frac1p |x'(t)|^p\bar{m} \d t = c(x_0,m_0,x_1,m_1)\, .
\]
In the case $|x_0-x_1|^p/p \geq 2\delta$, one obtains the same inequality by building a similar path, but transporting only an amount $\veps$ of mass in the second segment. 
\end{proof}

From this explicit minimal path cost, we obtain a semi-coupling formulation for $C_D$. We explain below how it is related to the problem of \emph{optimal partial transport}.


\begin{theorem}[Recovering optimal partial transport]\label{th:dynamic2partial}
For $(\rho_0,\rho_1)\in \mathcal{M}_+(\Omega)$, and considering the dynamic unbalanced optimal transport cost $C_D$ associated to the infinitesimal cost $f$ from~\eqref{eq:POT dynamic cost}, one has
\begin{equation}\label{partial OT Lagrangian}
C_D(\rho_0,\rho_1) = \delta(\rho_0(\Omega)+\rho_1(\Omega)) +  \inf_{\gamma \in \Gamma_{\leq}(\rho_0,\rho_1)}\int_{\Omega^2} \left( |x_1-x_0|^p/p - 2\delta \right) \d \gamma
\end{equation}
where the set of \emph{subcouplings} $\Gamma_{\leq}(\rho_0,\rho_1)$ is the subset of $\mathcal{M}_+(\Omega^2)$ such that the first and second marginals are upper bounded by $\rho_0$ and $\rho_1$, respectively.
\end{theorem}
\begin{proof}
Just for this proof, let us denote by $\tilde C$ the right hand side term. By Theorem~\ref{th: continuous static general}, $C_D$ admits a semi-coupling formulation involving the sublinear cost $\tilde c_f$ from equation~\eqref{eq: stat/dyn TV}. Let us consider optimal semi-couplings $(f_0\gamma,f_1\gamma)$ where $\gamma \in \mathcal{M}_+(\Omega^2)$ and $f_0,f_1 \in L^1(\gamma)$. Let $\bar \gamma = (f_0\wedge f_1)\bar \gamma|_{\mathcal D}$ where $\mathcal{D}=\{(x,y)\in \Omega^2\; ;\; |y-x|^p/p \leq 2\delta \}$. It holds  $\bar \gamma \in \Gamma_\leq (\rho_0,\rho_1)$ and 
\begin{align*}
C_K(\rho_0,\rho_1) &= \int_{\Omega^2} c(x_0,f_0(x_0,x_1)),x_1,f_1(x_0,x_1)) \d \gamma(x_0,x_1) \\
&= \int_{\Omega^2} \frac1p |x_1-x_0|^p \d \bar{\gamma} + \delta  |\gamma_0-\bar{\gamma}|(\Omega^2) + \delta|\gamma_1-\bar{\gamma}|(\Omega^2) \\
&= \delta \rho_0(\Omega)+ \delta \rho_1(\Omega) + \int_{\Omega^2} \left(\frac1p |x_1-x_0|^p -2\delta \right) \d \bar{\gamma} \geq \tilde C(\rho_0,\rho_1).
\end{align*}
For the opposite inequality, remark that the infimum defining $\tilde C$ is unchanged by adding the constraint that the sub-coupling $\gamma \in \Gamma_\leq(\rho_0,\rho_1)$ is concentrated on the set $\mathcal D$. For such a plan, let $\mu_i = \rho_i -(\Proj_i)_\#\gamma \in \mathcal{M}(\Omega)$ for $i\in \{0,1\}$ and define the pair of semi couplings
\[
\gamma_i = \gamma + \diag_\# (\mu_0\wedge \mu_1) + \diag_\# (\mu_i- \mu_0 \wedge \mu_1), \, i \in \{0,1\},
\]
where $\diag : x \mapsto (x,x)$ maps $\Omega$ to the diagonal in $\Omega^2$.  One has
\[
J_K(\gamma_0,\gamma_1) = \delta \rho_0(\Omega)+ \delta |\rho_1|(\Omega) + \int_{\Omega^2} \left(\frac1p |x_1-x_0|^p -2\delta \right) \d \gamma \, ,
\]
and it follows $C_D(\rho_0,\rho_1) \leq \tilde C(\rho_0,\rho_1)$.
\end{proof}

The minimization problem over \emph{subcouplings} introduced in Theorem~\ref{th:dynamic2partial} is a formulation of the \emph{optimal partial transport} problem, a variant of optimal transport studied in~\cite{caffarelli2010free, figalli2010optimal}. It is proved in~\cite{caffarelli2010free} that for any choice of $\delta>0$ corresponds the choice of $m(\delta)$ such that $0\leq m(\delta)\leq\min\{\rho_0(\Omega),\rho_1(\Omega)\}$ for which the minimizers of~\eqref{partial OT Lagrangian} and the minimizers of
\[
\min \left\{ \int_{\Omega^2} |y-x|^p\d \gamma(x,y) \; ;\; \gamma \in \Gamma_\leq(\rho_0,\rho_1), \gamma(\Omega^2)=m(\delta) \right\}
\]
are the same. The variable $\delta$ is the Lagrange multiplier for the constraint of total mass and corresponds to the maximum distance over which transport can occur (this does not mean, however, that optimal plans for~\eqref{partial OT Lagrangian} are obtained by simply restricting classical optimal transport plans to a set of bounded distance from the diagonal). The function $m(\delta)$ cannot be inverted in general (think of atomic measures) but it is proved in~\cite[Cor. 2.11]{caffarelli2010free} that it can be inverted if $\rho_0$ or $\rho_1$ is absolutely continuous.

Let us finally show that $C_D$, and thus optimal partial transport, admits a formulation as an optimal transport problem with relaxed marginal constraints.  This show a close link between the problems treated in this section and the problems considered in~\cite{piccoli2013properties, piccoli2014generalized} (although definitions differ slightly).

\begin{proposition}
For $\rho_0,\rho_1 \in \mathcal{M}_+(\Omega)$, one has
\begin{multline}\label{eq:WTV}
C_D(\rho_0,\rho_1) = \inf_{\gamma \in \mathcal{M}_+(\Omega^2)} \int_{\Omega^2} \frac1p |y-x|^p\d \gamma(x,y) \\ + \delta |\rho_0-(\Proj_0)_\#\gamma|(\Omega)+\delta |\rho_1-(\Proj_1)_\#\gamma|(\Omega)
\end{multline}
\end{proposition}
\begin{proof}
In this proof, we denote by $\tilde C(\rho_0,\rho_1)$ the infimum on the right-hand side. Given Theorem~\ref{th:dynamic2partial}, it is sufficient to show that the value of the infimum is unchanged when adding the constraint $(\Proj_i)_\# \gamma \leq \rho_i$ for $i \in \{0,1\}$. To prove this, consider $\gamma \in \mathcal{M}_+(\Omega^2)$ and build $\bar \gamma$ such that $\bar \gamma \leq \gamma$ and $(\Proj_0)_\# \bar \gamma = \rho_0 \wedge (\Proj_0)_\# \gamma$. By construction, one has 
 \begin{align*}
 |\rho_0-(\Proj_0)_\# \gamma| - |\rho_0-(\Proj_0)_\# \gamma^*| &= |(\Proj_0)_\# \gamma-(\Proj_0)_\# \gamma^*| \\ &= |\gamma-\gamma^*| \\
 \intertext{and}
  |\rho_1-(\Proj_1)_\# \gamma| - |\rho_1-(\Proj_1)_\# \gamma^*| &\geq - |(\Proj_1)_\# \gamma-(\Proj_1)_\# \gamma^*| \\&= - |\gamma-\gamma^*|.
 \end{align*}
 By denoting $F$ the functional in \eqref{eq:WTV} written as a function of  a coupling, it holds
 \[
 F(\gamma)-F(\gamma^*) \geq \int_{\Omega^2} (|y-x|^p/p)\d (\gamma-\gamma^*) \geq 0 \, .
 \]
A similar truncation procedure for the other marginal leads to the result.
\end{proof}

We conclude the study of this model with an instantiation of duality formulas, which are direct corollaries of Propositions~\ref{prop: dynamic dual} and~\ref{prop: duality}.

\begin{proposition}\label{prop:dualpartial}
We have the dual formulation
\begin{align*}
C_K(\rho_0,\rho_1) &=  \!\!\!\! \sup_{(\phi,\psi) \in C(\Omega)^2} \int_{\Omega} \phi \, \d \rho_0 + \int_{\Omega} \psi \, \d \rho_1
\end{align*}
subject to, for all $(x,y)\in \Omega^2$, $ \phi(x) + \psi(y) \leq \frac1p |y-x|^p$ and $\phi(x), \psi(y) \leq \delta$. Equivalently, 
\begin{align*}
C_K(\rho_0,\rho_1) &= \!\!\!\! \sup_{\varphi \in C^1([0,1]\times \Omega)} \int_{\Omega} \varphi(1,\cdot) \d \rho_1 - \int_{\Omega} \varphi(0,\cdot) \d \rho_0 \, ,
\end{align*}
subject to $ |\varphi| \leq \delta $ and $ \partial_t \varphi + \frac{p-1}{p} |\nabla \varphi|^\frac{p}{p-1} \leq 0 $ .
\end{proposition}

\subsection[Static Formulations for WF]{Static Formulation of the Wasserstein-Fisher-Rao metric}
\label{sec:equivalence WF}

Our second and last example deals with the case of the metric $\WF$.

\begin{definition}[$\WF$ metric]
\label{def:WF}
For a parameter $\delta \in ]0,+\infty[$ consider the convex, positively homogeneous, l.s.c.\ function
\begin{equation}
f : \R \times \R^d \times \R \ni (\rho,\omega, \zeta) \mapsto
\begin{cases}
\frac12 \frac{|\omega|^2+ \delta^2\, \zeta^2}{\rho} & \tn{if } \rho>0 \, ,\\
0 & \tn{ if } (\rho,\omega,\zeta)=(0,0,0)\, , \\
+ \infty & \tn{ otherwise,} 
\end{cases} 
\end{equation}
and define, for $\rho_0,\rho_1 \in \mathcal{M}_+(\Omega)$,
\begin{equation}
	\label{eq:WF}
	\WF(\rho_0,\rho_1)^2 \eqdef 
		\inf_{(\rho,\omega,\zeta) \in \mathcal{CE}_0^1(\rho_0,\rho_1)} \int_{[0,1] \times \Omega}
		f\left( \dens{\rho}{\lambda},\dens{\omega}{\lambda},\dens{\zeta}{\lambda} \right)
		\d \lambda
\end{equation}
where $\lambda \in \mathcal{M}_+([0,1] \times \Omega)$ chosen such that $\rho$, $\omega$, $\zeta \ll \lambda$. Due to the 1-homogeneity of $f$ the integral does not depend on the choice of $\lambda$.
\end{definition}

We now show that $\WF$ admits a static formulation, which belongs to the class of models introduced in Section \ref{sec : general static problem}. 
For this model, the distance between weighted Dirac measures $\WF(m_0 \delta_{x_0}, m_1 \delta_{x_1})^2$ has been computed in~\cite{ChizatOTFR2015} and is given by
	\begin{align}
		\label{eq:GeneralizedCost:OTFR}
		c(x_0,m_0,x_1,m_1) & = 2 \delta^2 \left( \vphantom{\sum} m_0 + m_1 - 2\sqrt{m_0\,m_1} \cdot \trucos(|x_0-x_1|/(2 \delta))\, .
		\right) 
		\end{align}
where $\trucos : z \mapsto \cos(|z|\wedge \frac{\pi}{2})$. Note that it is possible to compute directly the minimal path cost for this model and one obtains~\eqref{eq:GeneralizedCost:OTFR} but where the definition of $\overline \cos$ is replaced by $\cos(|z|\wedge \pi)$. This shows that the convex regularization of the cost is a crucial step in Theorem~\ref{th: continuous static general}.
%
\begin{theorem}[Static formulation of the metric]
\label{th: continuous static}
Choosing the cost function \eqref{eq:GeneralizedCost:OTFR}, it holds
\begin{equation}
\label{eq : continuous static}
\WF^2(\rho_0,\rho_1) = \min_{(\gamma_0,\gamma_1) \in \Gamma(\rho_0,\rho_1)} J_K(\gamma_0,\gamma_1)\,.
\end{equation}
\end{theorem}
\begin{proof}
This is a particular case of Theorem \ref{th: continuous static general}: it is clear that when this theorem holds then the convex relaxation of the minimal path cost equals the distance between pairs of weighted Dirac measures.
\end{proof}
\begin{remark}
This theorem can be reformulated as
\begin{multline*}
	 \frac{1}{2\delta^2} \WF^2(\rho_0,\rho_1) 
		=  \rho_0(\Omega) + \rho_1(\Omega)\, + \\
		 \qquad \inf_{(\gamma_0,\gamma_1) \in \Gamma(\rho_0,\rho_1)} - 2 \int_{|y-x|<\pi} \cos(|y-x|/(2\delta)) \d (\sqrt{\gamma_0 \gamma_1})(x,y)  \, .
\end{multline*}
where $\sqrt{\gamma_0\gamma_1} \eqdef \left(\dens{\gamma_0}{\gamma} \dens{\gamma_1}{\gamma}\right)^\frac12 \gamma$ for any $\gamma$ such that $\gamma_0,\gamma_1 \ll \gamma$.
\end{remark}
\begin{corollary}[Dual formulations]
\label{cor : dual static}
It holds
\begin{align*}
\frac{1}{2\delta^2} \WF^2(\rho_0,\rho_1) \quad =
& \sup_{(\phi,\psi)\in C(\Omega)^2  }  \int_\Omega \phi(x) \d \rho_0 + \int_\Omega \psi(y) \d \rho_1 \\
\tn{subject to, $\forall (x,y)\in \Omega^2$: }\quad
& \phi(x) \leq 1\,, \quad \psi(y) \leq 1\, , \\
&(1-\phi(x))(1-\psi(y)) \geq \trucos^2\left(|x-y|/(2\delta)\right)\, .
\end{align*}
or, by the change of variables $u = -\log(1-\phi)$ and $v = -\log(1-\psi)$:
\begin{align*}
\frac{1}{2\delta^2} \WF^2(\rho_0,\rho_1) \quad =
& \sup_{(u,v)\in C(\Omega)^2  }  \int_\Omega (1-e^{-u(x)})\d \rho_0 + \int_\Omega (1-e^{-v(y)}) \d \rho_1 \\
\tn{subject to, $\forall (x,y)\in \Omega^2$: }\quad
&u(x)+v(y) \leq -\log \left(\trucos^2\left(|x-y|/(2\delta)\right) \right)\, .
\end{align*}
\end{corollary}
\begin{proof}
By direct computations we find that $c(x,\cdot,y, \cdot) = \iota^*_{Q(x,y)}$ with
\[
Q(x,y) = \left\{ (a,b)\in \R^2 : a,b\leq 1 \tn{ and } (1-a)(1-b)\geq \trucos^2\right(|y-x|/(2\delta)\left) \right\} \, 
\]
and apply Proposition \ref{prop: duality}.
\end{proof}

For the sake of completeness, we state and prove an ``optimal entropy-transport'' formulation of the $\WF$ metric\footnote{this formulation was not present in the preliminary version of this article and  first appeared in \cite{LieroMielkeSavareLong}.}. It is obtained as the dual problem of the second formula in Corollary \ref{cor : dual static}.
\begin{corollary}[Optimal Entropy-Transport formulation]
\label{cor : static log-entropic}
The $\WF$ metric admits an ``optimal entropy-transport'' formulation :
\begin{multline*}
\frac{1}{2\delta^2} \WF^2(\rho_0,\rho_1) \quad =
\min_{\gamma \in \mathcal{M}_+(\Omega^2)} \Big\{ 
\int_{\Omega^2} c_\ell(x,y) \d \gamma (x,y) \\
+  \KL\left((\proj_0)_\#\gamma| \rho_0\right) + \KL\left((\proj_1)_\# \gamma| \rho_1\right) \Big\}
\end{multline*}
where $c_\ell(x,y) \eqdef -\log \left(\trucos^2\left(|x-y|/(2\delta)\right) \right)$ and 
\[
\KL(\mu|\nu) \eqdef
\begin{cases}
\int_{\Omega} (s\log s -s +1) \d \nu & \text{if $\mu \ll \nu$ and $\dens{\mu}{\nu} = s$} \\
+\infty & \text{otherwise.}
\end{cases}
\]
Moreover, the minimum is attained.
\end{corollary}
\begin{proof}
Let us compute the dual of the second formulation of Corollary \ref{cor : dual static}. It can be rewritten
\[
\sup_{(u,v) \in C(\Omega)^2} - F_0(-u) - F_1(-v) - G(A(u,v))
\]
where $A(u,v)(x,y) \eqdef u(x) +v(y)$ for all $(x,y)\in \Omega^2$, $F_i(u) \eqdef \int (e^{u}-1)\d \rho_i$ and $G(w) \eqdef 0$ if $w(x,y)\leq c_\ell(x,y)$ for all $x,y$ and $+\infty$ otherwise. Note that $c_\ell$ is nonnegative : it is thus easy to find a couple $(u,v)\in C(\Omega)^2$ such that $G$ is continuous at $A(u,v)$ (take for instance $u=v=-1$). Thus the Fenchel-Rockafellar Theorem applies, and there is strong duality with
\[
\min_{\gamma \in \mathcal{M}(\Omega^2)} F_0^*((\proj_0)_\# \gamma) + F_1^*((\proj_1)_\# \gamma) + G^*(\gamma)\, .
\]
since the adjoint operator of $A: C(\Omega)^2 \to C(\Omega^2)$ is the operator $\mathcal{M}(\Omega^2) \to \mathcal{M}(\Omega)^2$ which maps a measure to its two marginals. Moreover, by direct computations, $G^*(\gamma) = \iota_{\geq 0}(\gamma) + \int c_\ell \d \gamma$. Finally, the expressions $F^*_0 = \KL(\cdot|\rho_0)$ and $F^*_1 = \KL(\cdot|\rho_1)$ are justified by Lemma~\ref{lem:dualitymeasures}.
\end{proof}


\subsection[Gamma-convergence of Static WF]{$\Gamma$-convergence of Static $\WF$}
\label{sec:StaticGamma}

In~\cite{ChizatOTFR2015} the limit of the growth penalty parameter $\delta \rightarrow \infty$ of $\WF$ is studied and related to classical optimal transport. Here we give the corresponding result for the static problems in terms of $\Gamma$-convergence \cite{GammaConvergenceBraides2002}. Recall that this implies both convergence of the optimal values as well as convergence of minimizers.
We now denote by $c_\delta$ the cost defined in \eqref{eq:GeneralizedCost:OTFR} to emphasize its dependency on $\delta$.

\begin{theorem}[(Almost) Classical OT as Limit of $\WF$]
	\label{th:static gamma convergence}
	Consider the following two generalized static optimal transport functionals:
	\begin{align}
		\label{eq:modified static WF}
		J_\delta(\gamma_0,\gamma_1) & = \int_{\Omega^2} c_\delta 
		\left( x,\dens{ \gamma_0}{ \gamma},y,\dens{ \gamma_1}{ \gamma}\right)
		\d \gamma(x,y)		
		 - 2\,\delta^2 (\sqrt{\gamma_0(\Omega^2)}-\sqrt{\gamma_1(\Omega^2)})^2 \\
		 \intertext{where as before $\gamma \in \mathcal{M}_+(\Omega^2)$ is any measure such that $\gamma_0, \gamma_1 \ll \gamma$ and the integral does not depend on $\gamma$ due to 1-homogeneity of $c_\delta$.}
		J_\infty(\gamma_0,\gamma_1) & = \begin{cases}
			0 & \text{if } \gamma_0=0 \text{ or } \gamma_1 = 0\,, \\
			\int_{\Omega^2} |x-y|^2 \d\gamma_0(x,y) \cdot  \frac{\sqrt{\alpha}}{2} & \text{if } \gamma_1 = \alpha \, \gamma_0 \text{ for some } \alpha>0 \,, \\
			\infty & \text{otherwise.}
			\end{cases}
	\end{align}
	Then $J_\delta$ $\Gamma$-converges to $J_\infty$ as $\delta \rightarrow \infty$.
\end{theorem}
\begin{remark}
	One has $\lim_{\delta \rightarrow \infty} \WF(\rho_0,\rho_1) = \infty$ if $\rho_0(\Omega) \neq \rho_1(\Omega)$. Consequently, to properly study the limit, we subtract the diverging terms in \eqref{eq:modified static WF}. Conversely, we slightly modify the classical OT functional, to assign finite cost when the two couplings are strict multiples of each other. The corresponding optimization problem is solved by computing the optimal transport plan between normalized marginals and then multiplying by the geometric mean of the marginal masses.
	The above result implies
	\begin{multline*}
		\lim_{\delta \to \infty} \WF(\rho_0,\rho_1)^2 - 2\delta^2 (\sqrt{\rho_0(\Omega)}-\sqrt{\rho_1(\Omega)})^2 \\ = W_2
		\left(
		\tfrac{\rho_0}{\rho_0(\Omega)},
		\tfrac{\rho_1}{\rho_1(\Omega)}
		\right)^2 \cdot \sqrt{\rho_0(\Omega) \rho_1(\Omega)}
	\end{multline*}
	where $W_2$ denotes the standard 2-Wasserstein distance w.r.t.~the transport cost function $c(x,y)=\tfrac{|x-y|^2}{2}$. In particular, if $\rho_0(\Omega)=\rho_1(\Omega)$ then
	\begin{align*}
		\lim_{\delta \to \infty} \WF(\rho_0,\rho_1)=W_2(\rho_0,\rho_1).
	\end{align*}
\end{remark}
%
\noindent The proof uses the following Lemma.
\begin{lemma}[Sqrt-Measure]
	\label{lemma:SqrtMeasure1}
Let $A \subset \R^n$ be a compact set. The function
\begin{align}
	\mc{M}_+(A)^2 \ni \mu=(\mu_1,\mu_2) \mapsto -\sqrt{\mu_1 \cdot \mu_2}(A)
\end{align}
is weakly* l.s.c.~and bounded from below by $-\sqrt{\mu_1(A)} \cdot \sqrt{\mu_2(A)}$. The lower bound is only attained, if $\mu_1=0$ or $\mu_2=0$ or $\mu_1 = \alpha \cdot \mu_2$ for some $\alpha > 0$.
\end{lemma}
\begin{proof}
	With $f(x) = (\sqrt{x_1}-\sqrt{x_2})^2/2$ and a reference measure $\nu \in \mc{M}_+(A)$ with $\mu \ll \nu$ we can write
	\begin{align*}
		-\sqrt{\mu_1 \cdot \mu_2}(A) & = \int_A f\left( \frac{\mu}{\nu} \right)\,\d\nu - \mu_1(A)/2 - \mu_2(A)/2
	\end{align*}
	Since $f$ is 1-homogeneous, the evaluation does not depend on the choice of $\nu$. As $f$ is convex, l.s.c., bounded from below, $A$ is bounded, and total masses converge, lower semi-continuity of the functional now follows from \cite[Thm.~2.38]{ambrosio2000functions} (see proof of Proposition \ref{prop:KMinimizers} for adaption to $\Omega$ closed).
		
	For the lower bound, let $\mu_1 = \lambda \cdot \mu_2 + \mu_{1,\perp}$ be the Radon-Nikod\'ym decomposition of $\mu_1$ w.r.t.~$\mu_2$. Then have
	\begin{align*}
		-\sqrt{\mu_1 \cdot \mu_2}(A) & = - \int_A \sqrt{\lambda}\, \d\mu_2
			\geq - \left( \int_A \lambda\, \d\mu_2 \cdot \mu_2(A) \right)^{\tfrac12}
			\geq -\left(\mu_1(A) \cdot \mu_2(A) \right)^{\tfrac12}
	\end{align*}
	where the first inequality is due to Jensen's inequality, with equality only if $\lambda$ is constant. The second inequality is only an equality if $\mu_{1,\perp}=0$.
\end{proof}

\begin{proof}[Proof of Theorem \ref{th:static gamma convergence}]
\textbf{Lim-Sup.}
For every pair $(\gamma_0,\gamma_1)$ a recovery sequence is given by the constant sequence $(\gamma_0,\gamma_1)_{n\in \N}$. The cases $\gamma_i=0$ for $i=0$ or $1$, and $\gamma_1 \neq \alpha \, \gamma_0$ for every $\alpha>0$ are trivial. Therefore, let now $\gamma_1 = \alpha\,\gamma_0$ for some $\alpha>0$.
We find
\begin{align*}
	J_\delta(\gamma_0,\alpha\,\gamma_0) & = \int_{\Omega^2} 2\,\delta^2 \left[
		(1+\alpha) - 2 \sqrt{\alpha}\,\trucos(|x-y|/(2\delta)) \right] \d\gamma_0(x,y) \\
		& \qquad - 2\,\delta^2\,\gamma_0(\Omega^2)\,(1-\sqrt{\alpha})^2 \\
	\intertext{Now use $\trucos(z) \geq 1-z^2/2$ to find:}
	& \leq \int_{\Omega^2} 4\,\delta^2\,\sqrt{\alpha} \frac{|x-y|^2}{8\,\delta^2}\, \d\gamma_0(x,y) \\
	& = J_\infty(\gamma_0,\gamma_1)
\end{align*}

\textbf{Lim-Inf.} For a sequence of couplings $(\gamma_{0,k},\gamma_{1,k})_{k \in \N}$ converging weakly* to some pair $(\gamma_{0,\infty},\gamma_{1,\infty})$ we now study the sequence of values $J_k(\gamma_{0,k},\gamma_{1,k})$. Note first, that $J_k$ is weakly* l.s.c.~since the integral part is l.s.c.~(cf.~Proposition \ref{prop:KMinimizers}) and the second term is continuous (total masses converge).

Since $\Omega$ is compact, there is some $N_1 \in \N$ such that for $k>N_1$, we have
\begin{align*}
	1-z^2/2 \leq \trucos(z) \leq 1-z^2/2+z^4/24 \qquad \text{for} \qquad z=|x-y|/(2\,k),\,x,y \in \Omega\,.
\end{align*}
Let now $k>N_1$, $(\gamma_0,\gamma_1) \in \mathcal{M}_+(\Omega^2)^2$ and $\gamma \in \mathcal{M}_+(\Omega^2)$ such that $(\gamma_0,\gamma_1) \ll \gamma$. Denote $A \eqdef \sqrt{\gamma_0(\Omega^2)} - \sqrt{\gamma_1(\Omega^2)}$. Then
\begin{align*}
	J_k(\gamma_0,\gamma_1) & = 2\,k^2 \left(
	\int_{\Omega^2} \left( \sqrt{\dens{\gamma_0}{\gamma}}-\sqrt{\dens{\gamma_1}{\gamma}} \right)^2 \d \gamma
	- A^2 \right) \nonumber \\
	& \quad + 4\,k^2\,\int_{\Omega^2} \frac{|x-y|^2}{8\,k^2}\,\sqrt{\dens{\gamma_0}{\gamma}\dens{\gamma_1}{\gamma}}\, \d\gamma(x,y) \\
	& \quad - I \cdot 4\,k^2\,\int_{\Omega^2} \frac{|x-y|^4}{24 \cdot (2\,k)^4} \,\sqrt{\dens{\gamma_0}{\gamma} \dens{\gamma_1}{\gamma}}\, \d\gamma(x,y)
\end{align*}
for some $I \in [0,1]$.

Since $\Omega$ is bounded, by means of Lemma \ref{lemma:SqrtMeasure1} and since the total masses of $\gamma_{i,k}$, $i=0,1$ are converging towards the total masses of $\gamma_{i,\infty}$ as $k \rightarrow \infty$, there is a constant $C>0$ and some $N_2 \geq N_1$ such that the coefficient for $I$ in the third line can be bounded by $C/k^2$ for $k > N_2$ for when calling with arguments $(\gamma_{0,k},\gamma_{1,k})$.

For the first line we write briefly $2\,k^2\,F(\gamma_0,\gamma_1)$. From Lemma \ref{lemma:SqrtMeasure1} we find that $F$ is weakly* l.s.c., $F\geq0$ and $F(\gamma_0,\gamma_1)=0$ if and only if $(\gamma_0,\gamma_1) \in \mc{S}$ with
\begin{equation*}
	\mc{S} = \left\{ (\gamma_0,\gamma_1) \in \mc{M}_+(\Omega^2)^2 \,\colon\, \gamma_0 = 0 \tn{ or } \gamma_1 = 0 \tn{ or } \gamma_1 = \alpha \cdot \gamma_0 \tn{ for some } \alpha > 0 \right\}\,.
\end{equation*}
It follows that for $N_2 < k_1 < k_2$ one has
\begin{align*}
	J_{k_2}(\gamma_{0,k_2},\gamma_{1,k_2}) & = J_{k_1}(\gamma_{0,k_2},\gamma_{1,k_2}) + 2\,(k_2^2-k_1^2)\,F(\gamma_{0,k_2},\gamma_{1,k_2}) + I \cdot C/k_1^2
\end{align*}
for some $I \in [-1,1]$. Now consider the joint limit:
\begin{align*}
	\liminf_{k \rightarrow \infty} J_k(\gamma_{0,k},\gamma_{1,k}) & \geq \liminf_{k \rightarrow \infty} J_{k_1}(\gamma_{0,k},\gamma_{1,k}) + \liminf_{k \rightarrow \infty} 2(k^2 - k_1^2)\,F(\gamma_{0,k},\gamma_{1,k}) - C/k_1^2 \\
	& \geq J_{k_1}(\gamma_{0,\infty},\gamma_{1,\infty}) + 2(k_2^2 - k_1^2)\,F(\gamma_{0,\infty},\gamma_{1,\infty}) - C/k_1^2
\end{align*}
for any $N_2 < k_1 < k_2$ (by using weak* l.s.c.~of $J_{k_1}$ and $F$ and non-negativity of $F$). Since $J_{k_1} > - \infty$ and $k_2$ can be chosen arbitrarily large, we find
\begin{equation*}
	\liminf_{k \rightarrow \infty} J_k(\gamma_{0,k},\gamma_{1,k}) = \infty \quad \text{for} \quad (\gamma_{0,\infty},\gamma_{1,\infty}) \notin \mc{S}\,.
\end{equation*}
By reasoning analogous to the lim-sup case (adding the $z^4$ term in the $\trucos$-expansion to get a lower bound and bounding its value as above) we find
\begin{equation*}
	\liminf_{k \rightarrow \infty} J_k(\gamma_{0,k},\gamma_{1,k}) \geq J_\infty(\gamma_{0,\infty},\gamma_{1,\infty}) \quad \text{for} \quad (\gamma_{0,\infty},\gamma_{1,\infty}) \in \mc{S}\,.
	\qedhere  
\end{equation*}
\end{proof}


\section*{Conclusion and Perspectives}

In this paper, we presented a unified treatment of unbalanced optimal transport that allows for both static and dynamic formulations. Our key findings are (i) a new class of dynamic unbalanced optimal transport problems, (ii) a new class of static optimal transport formulations involving semi-couplings, (iii) an equivalence between these static formulations and a class of dynamic formulations. 
We believe that a key aspect of this work is that the proposed static formulation opens the door to a new class of numerical solvers for unbalanced optimal transport. These solvers should leverage the specific structure of the cost $c$ considered for each application, a striking example being the $\WF$ cost~\eqref{eq:GeneralizedCost:OTFR}.


\section*{Acknowledgements}

The work of Bernhard Schmitzer has been supported by the Fondation Sciences Math\'ematiques de Paris. 
The work of Gabriel Peyr\'e has been supported by the European Research Council (ERC project SIGMA-Vision). 
The work of Fran\c{c}ois-Xavier Vialard has been supported by the CNRS (D\'efi Imag'in de la Mission pour l'Interdisciplinarit\'e, project CAVALIERI).
\\
We would like to thank Yann Brenier for stimulating discussions.

\bibliographystyle{spmpsci}
\bibliography{SecOrdLandBig,SecOrdLand,bibchizat,references}

\end{document}